\documentclass[epsfig,11pt,final]{article}
\usepackage{geometry}                % See geometry.pdf to learn the layout options. There are lots.
\usepackage{graphicx}
\usepackage{epsfig}
\usepackage{amssymb}
\usepackage{epstopdf}
\usepackage{epsf,graphicx,amsmath,amsfonts}
\usepackage[color,notref,notcite]{showkeys}
\title{A fully-discrete Semi-Lagrangian scheme for a first order mean field game problem}

\author{ E. Carlini \thanks{Dipartimento di Matematica, Sapienza Universit\`a di Roma \tt{(carlini@mat.uniroma1.it)}  +39 06 49913214.} \and F. J. Silva  \thanks{XLIM - DMI 
UMR CNRS 7252
Facult\'e des Sciences et Techniques,
Universit\'e de Limoges {\tt (francisco.silva@unilim.fr) +33 5 87506787}. The support 
 of the European Union under the ``7th Framework Program FP7-PEOPLE-2010-ITN Grant agreement number 264735-SADCO'' is gratefully acknowledged.}}

 \newcommand{\Z}{{\mathbb Z}}

\newcommand{\ds}{\displaystyle}
\newcommand{\db}{\mathbf{d}}

\newcommand{\diver}{{\rm{div}}}

\newtheorem{remark}{\textbf{Remark}}[section]
\newcommand\be{\begin{equation}}
\newcommand\ee{\end{equation}}
\newcommand\ba{\begin{array}}
\newcommand\ea{\end{array}}
\newcommand{\bean}{\begin{eqnarray*}}
\newcommand{\eean}{\end{eqnarray*}}

\def\half{\mbox{$\frac{1}{2}$}}

\newcommand{\NN}{\mathbb{N}}
\newcommand{\ZZ}{\mathbb{Z}}
\newcommand{\RR}{\mathbb{R}}

\def\dd{{\rm d}}

\def\A{\mathcal{A}}
\def\B{\mathcal{B}}
\def\C{\mathcal{C}}

\def\G{\mathcal{G}}

\def\P{\mathcal{P}}

\def\Z{\mathcal{Z}}
\newcommand{\Rd}{{\mathbb R^d}}
\newcommand{\Zd}{{\mathbb Z^d}}

\newcommand{\ov}[1]{\overline{#1}}

\def\weight(#1,#2){c_{#1,#2}}

\def\eps{\varepsilon}

\def\rar{\rightarrow}

\def\ds{\displaystyle}

\newcommand{\supp}{\mathrm{supp}\;}
\newenvironment{proof}[1][Proof]{\textbf{#1.} }{\ \rule{0.5em}{0.5em}}
\newtheorem{lemma}{\textbf{Lemma}}[section]
\newtheorem{theorem}{\textbf{Theorem}}[section]

\newtheorem{proposition}{\textbf{Proposition}}[section]

\newtheorem{definition}{\textbf{Definition}}[section]

\numberwithin{equation}{section}

\begin{document}

\maketitle
\begin{abstract} In this work we propose a fully-discrete Semi-Lagrangian scheme for a {\it first order mean field game system}. We prove that the resulting discretization admits at least one solution and, in the scalar case,  we prove a convergence result for the scheme. Numerical simulations and examples are also discussed.
\end{abstract}



\thispagestyle{plain}
%\markboth{}{A SL scheme for a first order MFG problem}

\section{Introduction}
%%model
Initiated by the seminal work of Aumann \cite{Aumann64}, models to study equilibria in games with a large number of players have become an important research line in the fields of Economics and Applied Mathematics. In this direction, Mean Field Games (MFG) models were recently introduced by  J-M. Lasry and P.-L. Lions in \cite{LasryLions06i,LasryLions06ii,LasryLions07} in the form of  a  new system of Partial Differential Equations (PDEs). Under some assumptions, the solution of this system  captures the main properties of Nash equilibria for differential games with a very large number of {\it identical ``small'' players}. For a survey of MFG theory and its applications, we refer the reader to   \cite{Cardialaguet10,GLL10} and the lectures of P-L. Lions at the Coll\`ege de France \cite{cursolions}. The evolutive PDE system introduced in  \cite{LasryLions06ii}, with variables $(v,m)$,  is of the form:
\small
\be\ba{rl}\label{MFGgeneral} 
-\partial_{t} v (x,t) -\sigma^2 \Delta v(x,t)+H(x, D v(x,t)) &  = F(x, m(t)),  \;  \;    \hbox{in }  \RR^d\times (0,T), \\[6pt]
\partial_{t} m (x,t) -\sigma^2 \Delta m(x,t)-\diver\big(  \partial_{p}H(x,D  v(x,t)) m(x,t) \big) &=0, \; \; \; \hbox{in } \RR^d\times (0,T), \\[6pt]
v(x,T)= G(x, m(T)) \quad \mbox{for } x \in \RR^{d}  &, \; \; m(0)=m_0  \in \P_{1},
\ea\ee
\normalsize
where $\sigma \in \RR$, $\P_1$ denotes the space of probability measures on $\Rd$ and $F: \RR^{d} \times \P_{1} \to \RR$, $G: \RR^{d} \times \P_{1} \to \RR$ and $H:Ê\RR^{d} \times \RR^d \to \RR$ are given functions. 
%in order to analytically approximate    Nash equilibria when a large numbers of players are present. 
%
%in the presence of an infinite numbers of agents. 
%by Huang, Caines and Malham\'e \cite{Huangcainesmal03} formulated by
The Hamiltonian $H$ is supposed to be convex with respect to the second variable $p$. An important feature of the above system is its {\it forward-backward structure}: We have a backward Hamilton-Jacobi-Bellman (HJB) equation, i.e. with  a terminal condition, coupled with a forward Fokker-Planck equation with initial datum $m_0$.

Under rather general assumptions, it can be proved that if $\sigma\neq 0$ then   \eqref{MFGgeneral}  admits {\it regular solutions} (see \cite[Theorem 2.6]{LasryLions07}). Based on this fact, finite differences schemes have been thoroughly analyzed   in the papers \cite{AchdouCapuzzo10,AchdouCapuzzoCamilli10,AchdouCapuzzoCamilli12}. When $H(x,p)$ is quadratic with respect to $p$, specific methods have been proposed in  \cite{gueant10,Lachapelle10}. 

%
%Applications of mean field games to economics  F
%
%
%
%
%
%This system can represent a differential games with a very large number of  indistinguishable player,  whose strategies are influenced by the mass of the other agents. The model is very attractive since it can be applied to describe many applicative fields  in economics, physics and biology (see \cite{GLL10, aime10}).\\

In this work, we are interested in the numerical analysis of the {\it first order case} ($\sigma=0$) with quadratic Hamiltonian $H(x,p)=\half |p|^{2} $. In this case, system  \eqref{MFGgeneral} takes the form
\small
 \be\ba{rcl}\label{MFG}  -\partial_{t} v (x,t) +\half |D v(x,t)|^{2}  & =& F(x, m(t)),  \hspace{0.2cm}  \hbox{in }  \RR^d\times (0,T),  \\[6pt]
 \partial_{t} m (x,t) -\diver\big(D v(x,t) m(x,t) \big)  &=& 0,\quad \hbox{in } \RR^d\times (0,T), \\[6pt]
  v(x,T)= G(x, m(T)) \quad \mbox{for } x \in \RR^{d} & ,  & m(0)=m_0  \in \P_{1}.
\ea\ee
\normalsize
The second equation (i.e. the Fokker-Planck equation with $\sigma=0$) is called the {\it continuity equation} and   describes the transport of the initial measure $m_{0}$ by the flow induced by  $-Dv(\cdot, \cdot)$. 
When $F$ and $G$ are non-local  and {\it regularizing} operators (see \cite{LasryLions07}), the existence of a solution $(v,m)$ of \eqref{MFG} can be proved by a fixed point argument (see \cite{Cardialaguet10,cursolions}).
However, the numerical approximation  of $(v,m)$   is very challenging since, besides the   forward-backward structure of \eqref{MFG},  we can expect only Lipchitz regularity for $v$ and $L^{\infty}$ regularity for $m$ (see e.g.  \cite{Cardialaguet10}). 

%previuous works
Although several numerical methods have been analyzed for  each one of the  equations  in \eqref{MFG} (see e.g. the monographs  \cite{falconeferretilibro,Sethianbook,Osherbook} and the references therein for the HJB equation and \cite{PiccoliTosin,TosinFrasca} for the continuity equation), when the coupling between both equations is present,  the authors are aware only of references \cite{GosseJames02}, for the scalar case $d=1$, and \cite{AchdouCamilliCorrias11}, for the multidimensional case. However, in both references the structure of the system is {\it forward-forward}, i.e. both equations have initial conditions. This fact  completely changes the theoretical and numerical analysis of the problem. As a matter of fact,  for example  in  \cite{AchdouCamilliCorrias11}, the key property for convergence result of the proposed numerical scheme is a {\it one side Lipschitz condition} for $Dv(\cdot, \cdot)$ of the form:
\be\label{wewewepprpwrpwrw} \exists C>0 \; \; \mbox{such that } \forall t \in [0,T], \; \; \langle Dv(x, t)-Dv(y, t), x-y \rangle  \geq  -C |x-y|^{2}.\ee
By the results in \cite{PuopadRascle}, condition \eqref{wewewepprpwrpwrw} assures the stability of the so-called Fillipov characteristics and of the   associated  {\it measure solutions} of the continuity equation, which are the key to obtain their convergence result. Unfortunately, in our case \eqref{wewewepprpwrpwrw} corresponds to the semiconvexity  of $v$, which does  not holds for an arbitrary time horizon $T$ (see \cite{CannSinesbook}). 

Our  line of research   follows the ideas in  \cite{camsilva12}, where a semi-discrete in time Semi-Lagrangian scheme is proposed  to approximate \eqref{MFG} and a convergence result is obtained.  However, since the space variable is not discretized, the resulting scheme cannot be simulated. In this paper we propose a fully-discrete  Semi-Lagrangian scheme for \eqref{MFG}  and we study its main properties. We prove that the fully-discrete problem admits at least one solution and, for the case $d=1$, we are able to prove the convergence of the scheme to a solution $(v,m)$ of  \eqref{MFG},  when the discretization parameters tend to zero in a suitable manner. The key point of the proof is a discrete semiconcavity property for the discretized solutions.  Let us point out that our approximation scheme is presented in a general dimension $d$ and   several properties are proved in this generality. However, since in general \eqref{wewewepprpwrpwrw} does not hold, uniform estimates in the $L^\infty$ norm for the solutions of the scheme seems to be unavoidable in order to prove the convergence (see \cite{Cardialaguet10} for similar arguments regarding the vanishing viscosity approximation of \eqref{MFG}). Since we are able to prove these bounds only for $d=1$, our  convergence result for the fully-discrete scheme is  valid only in this case. 
% difficoltˆ nella prova

%% sintesi
The  paper is organized as follows:
In Section \ref{notandassum} we state our main assumptions, we collect some useful properties about semiconcave functions and we recall the main existence and uniqueness results for \eqref{MFG}.
% In  
%Section \ref{sectionsemi}  we revisit the semi-discrete in time approximation of  \cite{camsilva12} and we improve some results, for example we prove uniform $L^\infty$ bounds for the solutions of the semi-discrete scheme, which slightly improves  the convergence result of \cite{camsilva12}.
 Section \ref{fulldisc} is devoted to  the fully-discrete scheme. We establish the main properties of the scheme and we prove our main results: The fully-discrete scheme admits at least one solution and, if $d=1$ and the discretization parameters tend to zero in a suitable manner, every limit point of the solutions of the scheme is a solution of \eqref{MFG}. Finally, in Section \ref{numericaltests} we display some numerical simulations in the case of one space dimension.

\section{Preliminaries}\label{notandassum}
\subsection{Basic assumptions and existence and uniqueness results  for \eqref{MFG}}
We denote by $\P_1$ the set of the probability measures $m$ such that $\int_{\RR^d} |x|dm(x)<\infty$. The set $\P_1$ is be endowed with
the Kantorovich-Rubinstein  distance
\be\label{distance} \db_{1}(\mu, \nu)= \sup\left\{ \int_{\RR^{d}} \phi(x) \dd [\mu-\nu](x) \ ; \ \phi:\RR^{d}\to \RR \hspace{0.3cm} \mbox{is 1-Lipschitz}\right\}.
\ee
Given a measure $\mu \in \P_{1}$ we denote by $\supp(\mu)$   its support. In what follows, in order to simplify the notation,  the operator $D$ (resp. $D^{2}$) will denote the derivative (resp. the second derivative) with respect to the space variable $x\in \RR^{d}$. We suppose that the functions $F, G: \RR^{d}\times \P_{1} \to \RR$ and the measure $m_{0}$, which are  the data of   \eqref{MFG},   satisfy the following assumptions:\medskip\\
\textbf{(H1)} $F$ and $G$ are continuous over $\RR^{d}\times \P_{1}$. \smallskip\\
 \textbf{(H2)} There exists a constant $c_{0}>0$ such that for any $m\in \P_{1}$
 $$  \|F(\cdot, m)\|_{C^{2}} +  \|G(\cdot, m)\|_{C^{2}} \leq c_{0},$$
 where $ \|f(\cdot)\|_{C^{2}}:=\sup_{x\in \RR^d}\{|f(x)|+|Df(x)|+|D^2f(x)|\}.$ \smallskip\\
% HIPîTESIS DE CONVEXIDAD \textbf{(H3) [Convexity assumption]}  For all $m\in \P_{1}$ the functions $F(\cdot,m)$ and $G(\cdot,m)$ are convex. \smallskip\\
 \textbf{(H3)} The initial condition $m_{0}\in \P_{1}$ is absolutely continuous with respect to the Lebesgue measure, with density still denoted by $m_{0}$, and satisfies   $\supp(m_0)\subset B(0,c_{1})$ and  $\|m_{0}\|_{\infty}\leq c_{1}$, for some $c_{1}>0$ .  \smallskip
 
 As a general rule in this paper, given an absolutely continuous  measure (w.r.t the Lebesgue measure in $\RR^{d}$) $m\in \P_{1}$, its density will   still be denoted by $m$.
%HOPOTESIS DE MONOTONIA \textbf{(H4) [Monotonicity assumption] } The following  conditions holds true
%\begin{align}
%    \int_{\RR^{d}} \left[F(x,m_1)-F(x,m_2)\right] \dd[m_1-m_2](x)>0 \quad \mbox{for all }    m_1,\,m_2\in\P_1,\ m_1\neq m_2 \label{H2a},\\
%    \int_{ \RR^{d} }  \left[G(x,m_1)-G(x,m_2)\right] \dd[m_1-m_2](x)>0 \quad \mbox{for all }  m_1,\,m_2\in\P_1,\,m_1\neq m_2\label{H2b}.
%\end{align}\medskip
 %\begin{itemize}
%\item [$(\mathbf{H_1}$)] $H: \RR^d\times \RR^d\to \RR$, is $\cC^0$,
%and coercive in $p$, uniformly with respect to $x$.
%\item
%[($\mathbf{H_2}$)] \item [($\mathbf{H_3}$)] .....
%\end{itemize}
Let us recall the definition of a solution $(v,m)$   of \eqref{MFG} (see \cite{LasryLions06ii, LasryLions07}).\smallskip
\begin{definition} The pair $(v,m) \in W^{1,\infty}_{loc}(\RR^{d}\times [0,T])\times L^{1}(\RR^{d}\times (0,T))$ is a solution of \eqref{MFG} if the first equation  is satisfied  in the viscosity sense, while the second one  is satisfied in the distributional sense. More precisely,  for every $\phi \in \C_{c}^{\infty}\left((\RR^{d} \times [0,T) \right)$
\be\label{defsoldistribucioncompleta}\int_{\RR^{d}} \phi(x,0) m_{0}(x)\dd x + \int_{0}^{T}\int_{\RR^{d}}\left[ \partial_{t} \phi(x,t)- \langle D v(x,t), D \phi(x,t)\rangle\right]m(x,t) \dd x \dd t =0.\ee 
\end{definition}

\begin{remark} Classical arguments (see e.g. \cite{Ambrosiogiglisav}) imply that \eqref{defsoldistribucioncompleta} is equivalent to 
\be\label{defsoldistribucionsoloenespacio}  \int_{\RR^{d}} \phi(x) m_{0}(x)\dd x -\int_{0}^{t} \int_{\RR^{d}} \langle D v(x,s), D \phi(x)\rangle    m(x,s) \dd x \dd s = 0,\ee
 for all  $t \in [0,T]$ and  $\phi \in C_{c}^{\infty}(\RR^{d})$.
\end{remark}\smallskip

The following existence  result is proved in    \cite{cursolions,Cardialaguet10}. \smallskip
\begin{theorem}\label{existenciacasocont}  Under {\rm \textbf{(H1)}-\textbf{(H3)}} there exists at least a solution $(v,m)$ of \eqref{MFG}.
\end{theorem} \smallskip

A uniqueness result can be obtained assuming  \medskip\\
\textbf{(H4)} The following  {\it monotonicity} conditions hold true
\be\ba{l}
    \int_{\RR^{d}} \left[F(x,m_1)-F(x,m_2)\right] \dd[m_1-m_2](x)\geq 0 \quad \mbox{for all }  \ m_1, m_2 \in \P_{1}\\ [6pt]
    \int_{ \RR^{d} }  \left[G(x,m_1)-G(x,m_2)\right] \dd[m_1-m_2](x)\geq 0 \quad \mbox{for all } \,m_1, m_2 \in \P_{1}.
\ea \ee
We have (see   \cite{cursolions,Cardialaguet10}): \smallskip
\begin{theorem}\label{unicidadcasocont} Under {\rm \textbf{(H1)}-\textbf{(H4)}}  system \eqref{MFG} admits a unique solution  $(v,m)$.
\end{theorem}
%A additional regularity result can be obtained under the following assumption:\smallskip\\
%\textbf{(H5) [Convexity assumption]}  For all $m\in \P_{1}$ the functions $F(\cdot,m)$ and $G(\cdot,m)$ are convex. \smallskip\\
%We have 
%\begin{theorem}\label{existenciayunicidadcasoconvexo} Under  {\rm \textbf{(H1)}-\textbf{(H5)}} the solution $v$ of  \eqref{MFG} belongs to $C^{2,1}(\RR^{d}\times [0,T])$. Moreover, if $m_{0}$ is regular enough, then .... 
%\end{theorem}
%\begin{proof} See the Appendix. 
%\end{proof}
\subsection{Standard semiconcavity results}
In the proof of Theorem \ref{existenciacasocont}, as well as in the the proof of our main results, the concept of semiconcavity plays a crucial role. For a complete account of the theory and its applications to the solution of HJB equations, we refer the reader to the book \cite{ CannSinesbook}. \smallskip
\begin{definition} We say that $w:\RR^{d} \to \RR$ is semiconcave with constant $C_{conc}>0$ if  for every $x_{1}, x_{2} \in \RR^{d}$, $\lambda \in (0,1)$ we have
\be\label{desigdefinsemicon}w( \lambda x_{1} + (1-\lambda) x_{2}) \geq \lambda w(x_{1}) + (1-\lambda) w(x_{2}) - \lambda (1-\lambda) \frac{C_{conc}}{2} |x_{1}- x_{2}|^{2}.\ee
A function $w$ is said to be semiconvex if $- w$ is semiconcave.
\end{definition}
\smallskip

Recall that for $w:\RR^{d} \to \RR$ the super-differential $D^{+} w(x)$ at $x\in \RR^{d}$ is defined as
\be\label{definicionsuperdiferencial} D^{+}w(x):= \left\{ p \in \RR^{d} \ ; \ \limsup_{y\to x} \frac{ w(y)-w(x) - \langle p, y- x\rangle}{|y-x|} \leq 0 \right\}. \ee
We collect in the following Lemmas some useful properties of semiconcave functions (see \cite{ CannSinesbook}).\smallskip
\begin{lemma}\label{equivdefsemicon} For a function $w:\RR^{d}\to \RR$, the following assertions are equivalent:\smallskip\\
{\rm(i)}  The function $w$ is semiconcave, with constant $C_{conc}$.\smallskip\\
{\rm(ii)} For all $x,y\in \RR^{d}$, we have
$$ w(x+y)+ w(x-y) -2w(x) \leq C_{conc} |y|^{2}.$$
{\rm(iii)} For all $x$, $y \in \RR^{d}$  and $p\in D^{+}w(x)$, $q\in D^{+}w(y)$
\be\label{relacionmonotoniasemiconvex} \langle q-p, y-x \rangle \leq C_{conc}|x-y|^{2}.\ee
{\rm(iv) }Setting $I_{d}$ for the identity matrix, we have  that  $D^{2}w \leq C_{conc} I_{d}$ in the sense of distributions.
\end{lemma} \smallskip
\begin{lemma}\label{propsemiconcgenerica1} Let $w:\RR^{d} \to \RR$ be semiconcave. Then: \smallskip\\
{\rm(i) }  $w$ is locally Lipschitz. \smallskip\\
{\rm(ii)} If $w_{n}$ is a sequence of semiconcave functions (with the same semiconcavity constant) converging pointwise to  $w$, then the convergence is locally uniform and $Dw_{n}(\cdot)\to Dw(\cdot)$   a.e. in $  \RR^{d}$.
 \end{lemma}

\subsection{Representation formulas for the solutions of the HJB  and the continuity equations}
Let $\mu \in C([0,T]; \P_{1})$ be given and let us denote by $v[\mu]$ for  the unique viscosity solution of  
\be\label{hjbmu}\left.\ba{rcl} -\partial_{t} v (x,t) +\half |D v(x,t)|^{2}  & =& F(x, \mu(t)),  \hspace{0.2cm}  \hbox{in }  \RR^d\times (0,T),  \\[6pt]	
										v(x, T) & =& G(x,\mu(T)) \hspace{0.3cm}  \hbox{in }  \RR^{d}.\ea\right\}
\ee

Under assumptions \textbf{(H1)-(H2)}, standard  results  (see e.g.  \cite{BardiCapuzzo96}) yield that  for each $(x,t) \in \RR^{d}\times [0,T]$, the following representation formula for  $v[\mu](x,t)$ holds true \small
$$\left.\ba{l} v[\mu](x,t)=  \ds \mbox{inf}_{\alpha \in L^{2}([t,T]; \RR^{d})} \int_{t}^{T} \left[ \half |\alpha(s)|^{2} + F(X^{x,t}[\alpha](s), \mu(s)) \right] \dd s \\[8pt]
\hspace{4cm}+ G( X^{x,t}[\alpha](T), \mu(T)),\\[6pt]
\mbox{where } \hspace{1.4cm} X^{x,t}[\alpha](s):= x - \int_{t}^{s} \alpha(r) \dd r \hspace{0.5cm} \mbox{for all } s \in [t,T].\ea\right\} \eqno{(\mathcal{CP})^{x,t}[\mu]}  
$$ \normalsize
We set $\A^{x,t}[\mu]$ for the set of {\it optimal controls} $\alpha$ of  $(\mathcal{CP})^{x,t}[\mu]$, i.e. for the set of solutions of $(\mathcal{CP})^{x,t}[\mu]$. Classical arguments imply that for all $(x,t)$ the set  $\A^{x,t}[\mu]$ is non empty.

We now collect some important well known properties of problem $(\mathcal{CP})^{x,t}[\mu]$ (see e.g. \cite{CannSinesbook,Cardialaguet10}).
\begin{proposition}\label{propiedadesdelavalue} Under {\bf (H1)-(H2)}, The value function $v[\mu]$ satisfies the following properties:\smallskip\\
{\rm(i)} We have that $(x,t) \to v[\mu](x,t)$ is  Lipchitz, with a Lipschitz constant independent of  $\mu$. \smallskip\\
{\rm(ii)} For all $t\in [0,T]$ the function $  v[\mu](\cdot, t) \in \RR$ is semiconcave,  uniformly with respect to  $\mu$.  \smallskip\\
{\rm(iii)} There exists a constant $c_{2}>0$ (independent of $(\mu,x,t)$) such that
$$ \|  \alpha \|_{L^{\infty}([t,T];\RR^{d})}\leq c_{2} \hspace{0.3cm} \mbox{for all } \ \alpha \in \A^{x,t}[\mu].$$
{\rm(iv)} For all  $(x,t)$  and $\alpha \in \A^{x,t}[\mu]$, we have that 
\be\label{repredermascasocontinuo} \alpha(t) \in D^{+} v [\mu](x,t) .\ee
{\rm(v)} For all $t\in [0,T]$ the function $v[\mu](\cdot, t)$ is differentiable at $x$ iff there exists $\alpha \in  \A^{x,t}[\mu]$ such that $\A^{x,t}[\mu]= \{ \alpha\}$. In this case, we have that 
\be\label{repredermascasocontinuo2}  D v [\mu](x,t)= \alpha(t).\ee
{\rm(vi)}  For every $s\in (t, T]$ and $\alpha \in  \A^{x,t}[\mu]$,   we have that  $v[\mu](\cdot, \cdot)$ is differentiable at $(X^{x,t}[\alpha](s),s)$.
%{\rm(v)}  The function $(x,t) \to v[\mu](x,t)$ is differentiable at almost every $(x,t)$, and satisfies
%\be\label{repredermascasocontinuo} \alpha^{x,t}[\mu](t)= D v [\mu](x,t) .\ee
\end{proposition}
\medskip 

Now, we define {\it a measurable selection of optimal flows}, i.e. of optimal trajectories for the family of problems $\{ (\mathcal{CP})^{x,t}[\mu] \; ; \; (x,t) \in \RR^{d}\times [0,T]\}$. Classical arguments (see \cite{Cardialaguet10,AubinFran90}) show  that the multivalued map $(x,t)\to \A^{x,t}[\mu]$,  admits a measurable selection $\alpha^{x,t}[\mu](\cdot)$. Given $(x,t)$ the {\it flow} $\Phi[\mu](x, t, \cdot)$ is defined as 
\be\label{definicionphicontinua}  \Phi[\mu](x,t,s):= x -\int_{t}^{s} \alpha^{x,t}[\mu] (r) \dd r \hspace{0.4cm} \mbox{for all } s\in [t,T].
\ee
By Proposition \ref{propiedadesdelavalue}{\rm(v)}-{\rm(vi)}, omitting the dependence on $\mu$ for notational convenience, $\Phi(x,t,\cdot)$ satisfies
\be\label{ecuacionsatisfporphi}\left.\ba{rcl}  \frac{\partial}{\partial s} \Phi(x,t,s)&=&   -D v[\mu](\Phi(x,t,s),s)  \hspace{0.4cm} \mbox{for } s\in (t,T),\\[6pt]
												 \Phi (x,t,t)&=&   x.\ea\right\}
\ee
For all $t\in [0,T]$, let us  define $m[\mu](t)$ as  the initial measure $m_{0}$ {\it transported by the flow} $\Phi[\mu]$. More precisely, 
\be\label{RPxCE} m[\mu](t):= \Phi[\mu](\cdot,0,t) \sharp m_{0}, \ee
i.e.  $$m[\mu](t)(A)= m_{0}\left( \Phi[\mu]^{-1}(\cdot,0,t)(A)\right) \hspace{0.3cm} \mbox{for all } A\in \B(\RR^{d}),$$
or equivalently, for   all bounded and continuous $\phi: \RR^{d}\to \RR$, 
$$ \int_{\RR^{d}} \phi(x) \dd \left[m[\mu](t)\right](x) = \int_{\RR^{d}} \phi\left(\Phi[\mu](x,0,t) \right)\dd m_{0}(x).$$
Since $\Phi[\mu](x,\cdot, \cdot)$ satisfies the {\it semigroup property}, omitting the dependence on $\mu$ for simplicity,
$$ \Phi(x,s,t)= \Phi( \Phi(x,s,r),r,t) \hspace{0.4cm} \mbox{for all } r\in [s,t],$$
we easily check that 
\be\label{transportcomp} m[\mu](t):= \Phi(\cdot, r,t) \sharp \left[ \Phi(\cdot,0,r) \sharp m_{0} \right]=\Phi(\cdot, r,t) \sharp \left[m[\mu](r)\right]    \hspace{0.4cm} \mbox{for all } r\in [s,t].\ee
The fundamental result is the following  \smallskip
\begin{proposition}\label{caracterisnosecruzancont} There exists a constant $c_{3}>0$ (independent of $(\mu, x,y,r,t)$),  such that 
$$| \Phi(x, r,t)-\Phi(y, r,t)| \geq c_{3}|x-y| \hspace{0.3cm} \mbox{ {\rm for all} }\hspace{0.3cm} 0\leq r\leq t, \hspace{0.3cm} \mbox{{\rm and} }\hspace{0.3cm} x,y \in \RR^{d}.$$
\end{proposition}
The key in the proof of the above Proposition  (see e.g.  \cite[Lemma 4.13]{Cardialaguet10}) is   the semiconcavity of $v[\mu](\cdot,t)$, which is uniform w.r.t $\mu$,  and   Gronwall's Lemma. As a consequence we have that (see e.g. \cite[Theorem 4.18 and Lemma 4.14]{Cardialaguet10}) \smallskip
\begin{theorem}\label{expresionentermtransp}  We have that $m[\mu](\cdot)$ is the unique solution (in the distributional sense)  of  
\be\label{eqcontimu}\left.\ba{rcl} \partial_{t} m (x,t) -\diver\big(D  v[\mu](x,t) m(x,t) \big)  &=& 0,\quad \hbox{in } \RR^d\times (0,T), \\[6pt]	
										m(x,0) & =& m_{0}(x) \hspace{0.3cm}  \hbox{in }  \RR^{d}.\ea\right\}
\ee
Moreover,  there exists a constant $c_{4}>0$, independent of $\mu$, such that $m[\mu]$ satisfies the following properties:\smallskip\\
{\rm(i)}  For all $t\in [0,T]$, the measure $m[\mu](t)$ is absolutely continuous {\rm(}with density still denoted by $m[\mu](t)${\rm)},  has a support in $B(0,c_{4})$  and $\|m[\mu](t)\|_{\infty}\leq c_{4}$.\smallskip\\
{\rm(ii)} For all $t,t' \in [0,T]$, we have that
$$ \db_{1}( \mu(t), \mu(t')) \leq c_{4}|t-t'|.$$
\end{theorem}\smallskip
\begin{remark} In the proof of the above result (see \cite{Cardialaguet10}) Proposition \ref{caracterisnosecruzancont}  is crucial in order to show that the transported measure  $m[\mu](\cdot)$ is  absolutely continuous, as  $m_{0}$, and its density remains uniformly bounded in $L^{\infty}( \RR^{d})$.
\end{remark}\smallskip

Theorem \ref{expresionentermtransp}    {\rm(i)}-{\rm(ii)} implies   that $m[\mu](\cdot) \in C([0,T]; \P_{1})$. 
We thus see that \eqref{MFG} is equivalent to find $m\in C([0,T]; \P_{1})$, such that 
$$ m(t)= \Phi[m](\cdot,0,t) \sharp m_{0} \hspace{0.5cm} \mbox{for all } t\in [0,T]. \eqno{\rm(MFG)}$$ 
\section{The fully-discrete scheme}\label{fulldisc}
Given $h,\rho>0$, we consider a $d$ dimensional lattice $\mathcal{G}_{\rho}:=\{ x_i=i \rho,\; i\in \ZZ^d\}$ and a time-space grid $\mathcal{G}_{\rho,h}:= \mathcal{G}_{\rho}\times \{ t_k \}_{k=0}^{N}$, where $t_k=k h$ ($k=0,\hdots, N$) and $t_{N}=Nh=T$. We set  $B(\mathcal{G}_{\rho})$ and   $B(\mathcal{G}_{\rho,h})$ for the space of bounded functions defined on $\mathcal{G}_{\rho}$  and  $\mathcal{G}_{\rho,h}$, respectively. Given $f\in  B(\mathcal{G}_{\rho})$ and $g\in B(\mathcal{G}_{\rho,h})$ we will use the notation
$$ f_{i}:= f(x_{i}), \; \; \hspace{0.4cm} g_{i,k} := g(x_{i}, t_{k})  \hspace{0.4cm} \mbox{for all  $i \in \ZZ^{d}$ and $k=0,\hdots, N$}. $$ 
Let us consider the  $\mathbb{P}_1$ basis $\{ \beta_i \; ; \; i\in \ZZ^{d}\}$, where the function $\beta_{i}: \RR^{d} \to \RR$ is defined by $\beta_{i}(x):=\big[1-\frac{\|x-x_i\|_1}{\rho}\big]_+:= \max\{1-\frac{\|x-x_i\|_1}{\rho},0\}$. Denoting by $e_1, \hdots, e_d$   the canonical base of $\RR^{d}$,  it is easy to verify that $\beta_i(x)$ is continuous with compact support contained in $Q(x_i):=[x_{i}-\rho e_1,x_{i}+ \rho e_1]\times \dots  \times [x_{i}- \rho e_d,x_{i}+ \rho e_d]$, $0\leq \beta_i \leq 1$, $\beta_i(x_j)=\delta_{i j}$  (the Kronecker symbol) and $\sum_{i\in\ZZ^d}\beta_i(x)=1$. Let us consider the {\it interpolation operator}    $I[\cdot]:B(\mathcal{G}_{\rho})\to C_{b}(\RR^{d})$,  defined  by 
%%%%%%%%%%%%%%%%%%%%%%%%%%%%%%
\be\label{definterpolation}
I[f](\cdot):=\sum_{i\in\ZZ^d}f_i\beta_i(\cdot).
\ee
We recall a standard  estimate for $I$ (see e.g. \cite{Ciarlet,quartesaccosaleri07}).  Given $\phi \in C_{b}(\RR^{d})$, let us define $\hat{\phi} \in B(\mathcal{G}_{\rho})$ by $\hat{\phi}_i:=\phi(x_i)$ for all $i \in \ZZ^{d}$. We have that
\begin{equation}\label{www}
\sup_{x\in\RR^d}| I[\hat{\phi}](x)-\phi(x)|=O(\rho^\gamma),
\end{equation}
where   $\gamma=1$ if $\phi$  is Lipschitz and  $\gamma=2$  if $\phi\in C^2(\RR^{d})$ with bounded first and second derivatives. 

\subsection{The fully-discrete scheme for the HJB equation}  For a 
given $\mu \in C([0,T],\P_1)$,  we define  recursively  $v \in B(\G_{\rho,h})$ using the following {\it Semi-Lagrangian scheme} for \eqref{hjbmu}:
\begin{equation}\label{scheme-control}
  v_{i,k}=S_{\rho, h}[\mu](v_{\cdot,k+1},i,k) \quad   k=0, \hdots, N-1 \hspace{0.3cm} \mbox{and } \hspace{0.2cm}  v_{i,N}= G(x_{i}, \mu(t_N)), \\
\end{equation}
where $S_{\rho,h}[\mu]: B(\mathcal{G}_{\rho})\times \ZZ^{d} \times \{0,\hdots, N-1\} \to \RR$ is defined as 
%\be\label{defivdemu}\ba{rcl} 
%v^k_i&=&\inf_{\alpha\in \RR^{d}}\{\sum_{j \in \ZZ^{d}}\beta_j (x_i-h \alpha )v^{k+1}_i+\half h \alpha^2 )\}+hF(x_i,\mu(t_{k})), \hspace{0.3cm} i\in \mathbb{Z}^d,\; k=0,\hdots, N-1,\\ [6pt]
%  v^N_i&=&G(x_i,\mu(T)), \quad i\in \mathbb{Z}^d.\\ 
%\ea\ee
%
%Let us introduce the following short notation for the fully discrete scheme for the Hamilton Jacobi :
%\begin{equation}\label{scheme-control}
%  v^k_i=S_{\rho,h}(v^{k+1},m^{k};i) \quad \quad x_ j\in\mathcal{G}_{\rho},\quad n=N-1,\dots,0\\
%\end{equation}
%where 
\be\label{definicionesquema}
S_{\rho,h}[\mu](f,i,k):=\inf_{\alpha\in \RR^{d}} \left[ I[f](x_{i}-h\alpha) + \half h |\alpha|^{2} \right] + h F(x_{i}, \mu(t_{k})).
\ee
The following properties of      $S_{\rho,h}[\mu]$ are a straightforward consequence of the definition and assumptions {\bf (H1)} and {\bf (H2)}. \smallskip
\begin{lemma}\label{propiedadesbasicasesquema} The following assertions hold true:\smallskip\\
{\rm(i)} {\rm[The scheme is well defined]} There exists at least one $\alpha \in \RR^{d}$ that minimizes the r.h.s. of \eqref{definicionesquema}. Moreover, there exists $c_5>0$ such that $\sup_{i \in \ZZ^{d}, k=0,\hdots, N} | v_{i,k}| \leq c_{5}$. 
%if $p \in B(\mathcal{G}_{\rho})$ is bounded by a constant $c_{p}>0$, then the sequence 
%$$p_{N}:= G(p, \mu(T)), \hspace{0.4cm} p_{k}= S_{\rho,h}[\mu](p_{k+1},\ZZ^{d},k) \hspace{0.3cm} \mbox{for $k=0,\hdots, N-1$},$$
%is bounded by a constant that depends only on $c_{p}$ (i.e. it is independent of $\mu$). 
\smallskip\\
%{\rm(ii)}[Boundedness of the set of optimal controls] The set of $\alpha \in \RR^{d}$ which minimize  the r.h.s. of \eqref{definicionesquema} belongs to a compact set $K$, which is independent of $(\rho,h, \mu)$.\smallskip\\
{\rm(ii)}{\rm[Monotonicity]} For all $v, w \in  B(\mathcal{G}_{\rho})$ with $v\leq w$, we have that
\be\label{monotonia} S_{\rho,h}[\mu](v,i,k) \leq S_{\rho,h}[\mu](w,i,k) \hspace{0.5cm} \forall  \; i\in\ZZ^d,  \; k=0,\hdots, N.\ee
{\rm(iii)} For every $K \in \RR$ and  $w \in  B(\mathcal{G}_{\rho})$ we have 
\be\label{ck} S_{\rho,h}[\mu](w+K,i,n)= S_{\rho,h}[\mu](w,i,n)+K.\ee
{\rm(iv)}{\rm[Consistency]} Let  $(\rho_{n}, h_{n})\to 0$ (as $n\uparrow \infty$) and consider a sequence of grid points $(x_{i_{n}}, t_{k_n}) \to (x,t)$ and a sequence $\mu_{n} \in C([0,T]; \P_{1})$ such that $\mu_{n}\to \mu$. Then, for every 
$\phi \in  C^{1}\left(\RR^{d} \times [0,T)\right)$, we have\small
\be\label{Consistenzadebole}
\ba{ll}
\lim_{n\to \infty}  \frac{1}{h_n} \left[\phi(x_{i_n},t_{k_n})-S_{\rho_{n},h_{n}}[\mu_{n}](\phi_{k_{n+1}},i_n,k_n)\right] & =-\partial_t \phi(x,t)+\frac{1}{2}|D\phi(x,t)(x,t)|^2\\
																				\; & \;  \hspace{0.25cm} -F(x,\mu(t)).\ea\ee\normalsize
where $\phi_{k}=\{\phi(x_i,t_k)\}_{i\in\ZZ^d}$.
\end{lemma} \smallskip
%For any $f,g \in B(\mathbb{Z}^d)$ (bounded function defined on $\mathbb{Z}^d$) and for a give $m\in\SS$:
%\begin{description}
%\item[Monotonicity]
%\begin{equation}\label{Mon} {\mbox{if} }\; f_{j} \leq g_{j} \quad \mbox{then} \;
%S_{\rho,h}(f,m;j) \leq S_{\rho,h}(g,m;j) \;\mbox{with}\;j\;\mbox{such that} \;j\in\mathbb{Z}^2;
%\end{equation}
%\item[Invariance by addition of constants]
%\begin{equation}\label{ck}S_{\rho,h}(f+K;m,j)= S_{\rho,h}(f;m,j)+K\end{equation} for any $K\in\R;$
%\item[Consistency] Given two sequences $(\rho_m,h_m)$ and $(m_{\rho_m,h_m})$,  such that  $m_{\rho_m,h_m}\to \ov{m}$  and $(\rho_m,h_m)\to 0$ for $m\to \infty$. Let us call $m^{k_m}_i$ the coordinates associated to $ m_{\rho_m,h_m}$, 
% $ m_{\rho_m,h_m}(t_{k_m})= \frac{1}{\rho}\sum_{i=1}^{M_{\rho}}m^{k_m}_i \mathbb{I}_{E_i}(x) $ and let $(x_{j_m},t_{k_m})$ be grid point such that to $(x_{j_m},t_{k_m})\to(x,t)$ for $m\to \infty$ then for any smooth function $\phi(x,t)$ we have
%
%\end{description}
We define 
\be\label{definicionecontinuadevhrho}  v_{\rho,h}[\mu](x,t):= I[v_{\cdot, \left[ \frac{t}{h} \right]}](x) \hspace{0.5cm} \mbox{for all } \hspace{0.2cm} (x,t) \in \RR^{d}\times [0,T],
\ee
where we recall that $v_{i,k}$ is defined by \eqref{scheme-control}.   \smallskip

\begin{lemma}\label{propiedadesdelavmuhrhocontinua} For every $t\in [0,T]$, the following assertions hold true:\smallskip\\
{\rm(i) [Lipschitz property]} The function  $v_{\rho,h}[\mu](\cdot,t)$ is Lipschitz with constant independent of $(\rho,h,\mu,t)$.  \\[4pt]
{\rm(ii) [Weak semiconcavity]} There exists  $c_6>0$  independent of $(\rho,h,\mu,t)$ such  that for all $x$,$y \in \RR^d$ we have \small
\be\label{semiconcavidaddebilcondosf} v_{\rho,h}[\mu](x+y,t)  - 2v_{\rho,h}[\mu](x,t) +v_{\rho,h}[\mu](x-y,t) \leq c_{6} \left[|y|^2+\rho^2(E(x+y) + E(x-y))\right],
\ee \normalsize
where $E: \RR^{d}\to \RR$ is a nonnegative, continuous and bounded function vanishing in $\mathcal{G}_{\rho}$.
\end{lemma}\smallskip
\begin{proof}  By {\bf{(H2)}} we have  that $\| DG (\cdot, \mu(T))\|_{\infty} \leq c_{0}$  and so $I[G](\cdot, \mu(T))$ is $c_0$-Lipschitz. Thus, by  the \eqref{scheme-control} and \eqref{definicionecontinuadevhrho}, we get that    $v_{\rho,h}[\mu](\cdot,t_{N-1})$ is Lipschitz with constant $hc_{0} + c_{0}$. Iterating the argument, using  {\bf{(H2)}} for $F$, we get that $v_{\rho,h}[\mu](\cdot, t)$ is $c_0(1+T)$ Lipschitz for all $t\in [0,T]$.  The proof of the second assertion is provided e.g.  in \cite[Lemma 4.1]{AchdouCamilliCorrias11}. 
\end{proof}\smallskip
\begin{remark} Inequality \eqref{semiconcavidaddebilcondosf}  is a consequence of a {\it discrete}-semiconcavity property of $v_{i,k}$ (see e.g. \cite{Lintadmor01}).\end{remark}
\begin{theorem}\label{covergencefullydiscrete}  Let  $(\rho_{n}, h_{n})\to 0$ (as $n\uparrow \infty$) be such that $ \frac{\rho_{n}^{2}}{h_{n}} \to 0$. Then, for every sequence $\mu_{n}\in C([0,T]; \P_{1})$ such that $\mu_{n}\to \mu$ in $C([0,T]; \P_{1})$, we have that $v_{\rho_{n},h_{n}}[\mu_{n}]\to v[\mu]$ uniformly over compact sets.
%
% For any limit point $\ov{m}$ of $m^{\rho,h}$, when $(h,\rho)\to 0$ as long  as $\frac{\rho^2}{h}\to 0$, we have that
% $v_{\rho,h}[\bar m]$ converges locally uniformly to the unique solution $v[\bar m]$  of \eqref{MFG}.
\end{theorem}
\begin{proof}  Using assumption {\bf(H1)}, the proof  is a straightforward variation of the one in \cite{CFF10}, which is a revised proof of the result given in \cite{BS91}. However, for the sake of completeness we provide the details. 
%Let  us call $(h_m,\rho_m)$ a sequence such that $m_{h_m,\rho_m}\to \ov{m}$ for $m\to \infty$, where  $\ov{m}\in C([0,T],\P_1)$ and $(h_m,\rho_m)\to 0$. 
%Since $m_{h_m,\rho_m} \to \ov{m} $ in $C([0,T]; \P_{1})$ and $F$ is uniformly continuous we have  $F(x,m^{k_m})\to F(x,\ov{m}(\bar{t})$
%
%By Lemma \ref{Lemmaboundform_hro} we also have that $m_{h_m,\rho_m}\to \ov{m}$ in $L^\infty(\RR^2\times[0,T])$-weak-*.\\
For $(y,s)\in \RR^{d}\times[0,T]$,   set
$$ v^{*}(y,s):= \underset{n\to \infty} {\limsup_{(y',s')\to (y,s)}}  v_{\rho_{n},h_{n}}[\mu_{n}](y', s'), \hspace{0.5cm}  v_{*}(y,s):= \underset{n\to \infty}{\liminf_{(y',s')\to (y,s)}} v_{\rho_{n},h_{n}}[\mu_{n}](y', s').
$$
Let us prove that $v^{*}$ is a viscosity subsolution of
\be\label{ecuacioninetermediaria}\ba{rcl} -\partial_{t} v(x,t) + \half |D v(x,t)|^{2}& =& F(x, \mu(t)) \quad \mbox{for } (x,t) \in \RR^{d} \times (0,T), \\[6pt]
								v(x,T) &=&  G(x, \mu(T)) \quad \mbox{for } x \in \RR^{d}.\ea	
\ee
Let $(\bar{y}, \bar{s}) \in \RR^{d}\times (0,T)$  and $\phi \in C^{1}(\RR^{d}\times (0,T))$ be such that $v^{*}(\bar{y}, \bar{s})=\phi(\bar{y}, \bar{s})$ and $v^{*}-\phi$ has a global strict maximum    at  $(\bar{y},\bar{s})$. Since $v^*(\cdot, \cdot)$ is upper semicontinuous,   a standard argument in the theory of viscosity solutions implies that, up to some subsequence, there exists   $(y_{n},s_{n})\to (\bar{y},\bar{s})$, such that \small
$$(v_{\rho_{n},h_{n}}[\mu_{n}]-\phi)(y_{n}, s_{n})= \max_{(y,s) \in \RR^{d}\times (0,T)} (v_{\rho_{n},h_{n}}[\mu_{n}]-\phi)(y,s)$$ $$\hspace{0.3cm} \mbox{and } \hspace{0.2cm} (v_{\rho_{n},h_{n}}[\mu_{n}]-\phi)(y_{n}, s_{n}) \to (v^{*}-\phi)(\bar{y}, \bar{s})= 0. $$\normalsize
%
%there exist two sequences $({h}_m,\rho_m)\in \R^+\times \R^+$ and $(x_m,t_{m})\in \Rd \times{[0,T]}$, which are global maximum points for $v^{h_m,\rho_m}-\phi$, and as $m\to\infty$, up to susequence, are such that :\\
%$$
%({h}_m,\rho_m) \to 0, \quad (y_m,\tau_m)\to (\bar{x}, \bar{t}), \quad v_{h_m,\rho_m}(x_m,t_m)\to v^*(\bar{x},\bar{t}) \hspace{0.4cm} $$
Thus, for any $(y,s)\in  \RR^{d}\times [0,T]$  we have that
\begin{equation}\label{eq:1}
  v_{\rho_{n},h_{n}}[\mu_{n}](y,s) \leq \phi(y,s) +\xi_n, \hspace{0.4cm} \mbox{with } \hspace{0.4cm} \xi_n :=  (v_{\rho_{n},h_{n}}[\mu_{n}]-\phi)(y_{n},s_{n}) \to 0.
\end{equation}
%with $\xi_m =  (v_{h_m,\rho_m}-\phi)(x_m,t_m)$ (note that, since $v^*(x,t)=\phi(x,t)$, we have $\xi_m\to 0$).
%Then, for any $x$ and $t$ we have
%\begin{equation}\label{eq:1}
 % v_{h_m,\rho_m}(x,t) \leq \phi(x,t) + (v_{h_m,\rho_m}-\phi)(x_m,t_m)
%\end{equation}
%and in particular applying \eqref{eq:1} to $(x_m-ha,t_m+h)$ we get
%\begin{equation}\label{eq:2}
 % v_{h_m,\rho_m}(x_m-ha,t_m+h) -\phi(x_m-ha,t_m+h) \leq v_{h_m,\rho_m}(x_m,t_m)-\phi(x_m,t_m)
%\end{equation}
%with $\xi_m =  (v_{h_m,\rho_m}-\phi)(x_m,t_m)$ (note that, since $v^*(x,t)=\phi(x,t)$, we have $\xi_m\to 0$).
 Let  $k:= \NN \to \{0,\hdots, N-1\}$   be  such that $ s_{n} \in   [t_{k(n)}, t_{k(n)+1})$.  Evidently, we have that   $t_{k(n)} \to \bar{s}$.  By taking $y=x_{i}$, $i\in \ZZ^d$, and $s=t_{k(n)+1}$ in  \eqref{eq:1},  we get that 
%In general, $(x_m ,t_m)$ is not a grid node, by the way  there exists two sequence of cell grids
%such that $t_m \in [t_{k_m},t_{k_m+1})$ and 
%$x_m \in  Q(x_{j_m}):=[x_{j_m},x_{j_m}+h e_1)\times \dots  \times [x_{j_m},x_{j_m}+h e_d)$ (where $(e_i), i=1,\dots d$ is the canonical basis).
%Let us denote the set of indices of the nodes contained in $Q(x_{j_m})$ as $\mathcal{I}(x_m)$, we observe that only this nodes are involved in the reconstruction
%of $v_{\rho,h}$ in $x_m$.\\
%By the definition of $u^{h_m,\rho_m}$, there exist a $t_{n_m}$ such that $\tau_m \in [t_{n_m},t_{n_m+1})$ and $u^{h_m,\rho_m}(y_m,\tau_m)= u^{h_m,\rho_m}(y_m,t_{n_m})$.
\begin{equation}\label{eq:2a}
  v_{i,  k(n)+1}  \leq \phi(x_{i}, t_{k(n)+1}) +\xi_n \hspace{0.4cm} \mbox{for all } i\in \ZZ^{d}.
\end{equation}
Lemma \ref{propiedadesbasicasesquema}{\rm(ii)}-{\rm(iii)} implies that 
$$
S_{\rho_{n},h_{n}}[\mu_{n}](v_{\cdot,k(n)+1},i, k(n))\leq S_{\rho_{n},h_{n}}[\mu_{n}](\phi_{k(n)+1},i, k(n))  +\xi_n \hspace{0.4cm} \mbox{for all } i\in \ZZ^{d}.
$$ 
In particular,  using \eqref{scheme-control}, we get
$$
 v_{i ,k(n)}  \leq  S_{\rho_{n},h_{n}}[\mu_{n}](\phi_{k(n)+1}, i, k(n))  +\xi_n \hspace{0.4cm} \mbox{for all } i\in \ZZ^{d},
$$ 
which yields, by the definition of  $v_{\rho_{n},h_{n}}[\mu_{n}](y_{n},  s_n)$ in \eqref{definicionecontinuadevhrho}, 
$$
 v_{\rho_{n},h_{n}}[\mu_{n}](y_{n}, s_n) \leq \sum_{i \in \ZZ^{d}} \beta_{i}(y_{n}) S_{\rho_{n},h_{n}}[\mu_{n}](\phi_{k(n)+1},i, k(n))  +\xi_n.
$$ 
Now, recalling the definition of $\xi_n$,   we get
\begin{equation} \label{eq:2}
\phi(y_n,s_{n}) \leq \sum_{i \in \ZZ^{d}} \beta_{i}(y_{n}) S_{\rho_{n},h_{n}}[\mu_{n}](\phi_{k(n)+1},i, k(n)).
\end{equation}
We claim now that $\phi(y_{n},s_{n})=\phi(y_{n}, t_{k(n)})+O(h_{n}^2)$.
In fact, either $s_{n}=t_{k(n)}$ (and the claim obviously holds), or $s_n \in (t_{k(n)}, t_{k(n)+1})$. In the latter case, since $(v_{\rho_{n},h_{n}}-\phi)(y_{n},\cdot)$ has a maximum at $s_{n}$ and $v_{\rho_{n},h_{n}}$
 is constant in $(t_{k(n)}, t_{k(n)+1})$, then $\partial_t \phi(y_{n},s_{n})=0$ and   the claim follows from a Taylor expansion. Thus, by our  claim and \eqref{eq:2}, we have that
\begin{equation}\label{zzz}
\phi(y_{n},t_{k(n)}) \leq \sum_{i \in \ZZ^{d}} \beta_{i}(y_{n}) S_{\rho_{n},h_{n}}[\mu_{n}](\phi_{k(n)+1},i, k(n)) +o(h_{n}).
\end{equation}
Now, inequality \eqref{zzz}, estimate \eqref{www} and the fact that $\rho_{n}^{2}/h_{n} \to 0$ imply that
\begin{equation*}
 \lim_{n\to \infty} \sum_{i\in \ZZ^{d}} \beta_{i}(y_{n})
 \frac{\phi(x_{i},t_{k(n)})-S_{\rho_{n},h_{n}}[\mu_{n}](\phi_{k(n)+1},i, k(n))}{h_{n}} \leq 0.
\end{equation*}
Finally, by the consistency property in Lemma \ref{propiedadesbasicasesquema}{\rm(iv)}  we obtain that
\begin{equation*}
-\partial_t \phi(\bar{y},\bar{s})+\frac{1}{2} |D \phi (\bar{y},\bar{s})|^2-F(\bar{y},\mu(\bar{s})) \leq 0,
%\underline{F}(D \phi, D^2 \phi)(x,t)
\end{equation*}
which implies that $v^{*}$  is a subsolution of \eqref{eq:1}. The supersolution property for $v_{*}$ can be proved in a similar manner. Therefore, by a classical comparison argument,  $v_{\rho_{n},h_{n}}[\mu_{n}]$ converges locally uniformly to $v[\mu]$ in $\RR^{d} \times(0,T)$.
\end{proof}\medskip

%Note that for all $t \in [0,T]$, the function $v_{\rho,h}[\mu](\cdot, t)$ is not, in general,  differentiable at $x_{i}\in \mathcal{G}_{\rho}$. Thus,  we cannot use the useful characterizations of weak semiconcavity (see Lemma \ref{propiedadesbasicassemiconcavasdebiles}) for differentiable functions. Therefore, we will regularize $v_{\rho,h}[\mu](\cdot, t)$ with the usual convolution technique. 
 Let $\rho \in C^{\infty}_{c}(\RR^{d})$ with $\rho\geq 0$ and $\int_{\RR^{d}} \rho(x) \dd x=1$.  For $\eps>0$, we consider the {\it mollifier}  $\rho_{\eps}(x):= \frac{1}{\eps^{d}} \rho\left(\frac{x}{\eps}\right)$
and define 
\be\label{definicionregularizada} v^{\eps}_{\rho,h}[\mu](\cdot,t):= \rho_{\eps} \ast  v_{\rho,h}[\mu](\cdot,t) \hspace{0.5cm} \mbox{for all } \hspace{0.2cm} t\in [0,T].
\ee
Using Lemma \ref{propiedadesdelavmuhrhocontinua}{\rm(i)}  we easily check  the   estimates
\be\label{aproximacionuniformedelaconvolucion}\ba{rcl} \| v^{\eps}_{\rho,h}[\mu](\cdot,\cdot)-v_{\rho,h}[\mu](\cdot,\cdot)\|_{\infty} &=& \gamma \eps, \\[6pt]
											 \| D^{\alpha}v^{\eps}_{\rho,h}[\mu](\cdot,\cdot)\|_{\infty} &=& c_{\alpha}\eps^{1-|\alpha|} \ea
\ee
where $\gamma>0$ is independent of  $(\eps,\rho,h,\mu,t)$, $\alpha$ is a multiindex with $|\alpha|>0$ and $c_{\alpha}>0$ depends only on $\alpha$. We have: \smallskip
\begin{lemma}\label{propiedadesdelavmuhrhocontinuaconeps} For every $t\in [0,T]$, the following assertions hold true:\smallskip\\
{\rm(i)} {\rm [Lipschitz property]} The function  $v^{\eps}_{\rho,h}[\mu](\cdot,t)$ is Lipschitz with constant  $  d_{0}$ independent of $(\rho,h,\mu,t)$.  \\[4pt]
{\rm(ii)} {\rm [Semiconcavity]} There exists   $d_{1}>0$  independent of $(\rho,h,\eps,\mu,t)$, such that
\be\label{mamdmmmdaaaassss} \langle D^{2}v^{\eps}_{\rho,h}[\mu](x,t)y, y \rangle \leq d_1\left(1 + \frac{\rho^{2}}{\eps^3}\right)|y|^{2} \hspace{0.5cm} \forall \; x,y \in \RR^{d}.\ee 
\end{lemma}\smallskip
\begin{proof} Assertion {\rm(i)}  follows directly from the definition of $v^{\eps}_{\rho,h}[\mu](\cdot,t)$  and the corresponding result for $v_{\rho,h}[\mu](\cdot,t)$  in  Lemma  \ref{propiedadesdelavmuhrhocontinua}. Now, let us prove assertion {\rm(ii)}.  If $y=0$ the result is true so let us assume that $y\neq 0$ and set $\tau= \rho/|y|$.   Assertion {\rm(ii)} in Lemma  \ref{propiedadesdelavmuhrhocontinua}  implies the existence of $c_{7}>0$ (independent of $(\rho,h,\eps,\mu,t)$)  such that  
$$v^{\eps}_{\rho,h}[\mu](x+y',t)  - 2v^{\eps}_{\rho,h}[\mu](x,t) +v^{\eps}_{\rho,h}[\mu](x-y',t) \leq c_{7} \left(|y'|^2+\rho^2\right) \hspace{0.3cm} \forall \; y' \in \RR^{n}.$$
Setting $y'= \tau y $, we obtain   that 
\be\label{semiconcavidaddebilcondossdf} v^{\eps}_{\rho,h}[\mu](x+\tau y,t)  - 2v^{\eps}_{\rho,h}[\mu](x,t) +v^{\eps}_{\rho,h}[\mu](x-\tau y,t) \leq 2c_{7}|\tau|^{2} |y|^2.
\ee   
On the other hand,  by a Taylor expansion and taking $|\alpha|=4$ in the the second estimate in \eqref{aproximacionuniformedelaconvolucion}, we get, using the multiindex notation,  \small
$$\ba{rcl} v^{\eps}_{\rho,h}[\mu](x+\tau y,t) &\geq& v^{\eps}_{\rho,h}[\mu](x,t)+ \langle Dv^{\eps}_{\rho,h}[\mu](x,t), \tau y \rangle+ \half \langle D^{2}v^{\eps}_{\rho,h}[\mu](x,t)\tau y, \tau y \rangle \\[4pt]
\; & \; & + \sum_{|\alpha|=3} \frac{1}{\alpha ! } D^{\alpha}v^{\eps}_{\rho,h}[\mu](x,t)(\tau y)^{\alpha}  -\frac{1}{\eps^3}c'  |\tau y |^{4}, \\[4pt]
		 v^{\eps}_{\rho,h}[\mu](x-\tau y,t) &\geq& v^{\eps}_{\rho,h}[\mu](x,t)- \langle Dv^{\eps}_{\rho,h}[\mu](x,t), \tau y \rangle+ \half \langle D^{2}v^{\eps}_{\rho,h}[\mu](x,t)\tau y, \tau y \rangle\\[4pt]
\; & \; &- \sum_{|\alpha|=3} \frac{1}{\alpha ! } D^{\alpha}v^{\eps}_{\rho,h}[\mu](x,t)(\tau y)^{\alpha} - \frac{1}{\eps^{3}}c' |\tau y |^{4},\ea$$ \normalsize
where $c'>0$ is independent of $(\rho,h,\eps,\mu,t)$. Adding both inequalities, using \eqref{semiconcavidaddebilcondosf} and that $|\tau y | =\rho$, we get
$$ \langle D^{2}v^{\eps}_{\rho,h}[\mu](x,t)\tau y, \tau y \rangle \leq  \left(2c_{7}+ 2c'\frac{\rho^{2}}{\eps^{3}}\right)\tau^{2} |y |^{2}.$$
Dividing by $\tau^{2}$ and taking $d_{1}= \max\{ 2c_{7}, 2c' \}$ we get the result. 
\end{proof} \smallskip

As a consequence we obtain \smallskip
\begin{theorem}\label{covergencefullydiscreteconeps}  Let  $(\rho_{n}, h_{n},\eps_{n})\to  0$  be such that $ \frac{\rho_{n}^{2}}{h_{n}} \to 0$ and $\small \rho_{n}=O(\eps_{n}^{3/2})$\normalsize. Then, for every sequence $\mu_{n}\in C([0,T]; \P_{1})$ such that $\mu_{n}\to \mu$ in $C([0,T]; \P_{1})$, we have that $v^{\eps_{n}}_{\rho_{n},h_{n}}[\mu_{n}]\to v[\mu]$ uniformly over compact sets and $Dv^{\eps_{n}}_{\rho_{n},h_{n}}[\mu_{n}](x,t)\to Dv[\mu](x,t)$ at every $(x,t)$ such that $Dv[\mu](x,t)$ exists.
%
% For any limit point $\ov{m}$ of $m^{\rho,h}$, when $(h,\rho)\to 0$ as long  as $\frac{\rho^2}{h}\to 0$, we have that
% $v_{\rho,h}[\bar m]$ converges locally uniformly to the unique solution $v[\bar m]$  of \eqref{MFG}.
\end{theorem}
\begin{proof} The first assertion is a consequence of Theorem \ref{covergencefullydiscrete} and the uniform estimate \eqref{aproximacionuniformedelaconvolucion}.  Next, fix $x, y \in \RR^{d}$. Then, since $\rho_{n}^{2}/\eps^{3}_{n} \leq C$ for some $C>0$  (independent of $n$),  inequality \eqref{mamdmmmdaaaassss} implies the existence of $C'>0$ (independent of $n$) such that
\be\label{aadamammsssssaaa}v^{\eps_{n}}_{\rho_{n},h_{n}}[\mu_{n}](y,t)- v^{\eps_{n}}_{\rho_{n},h_{n}}[\mu_{n}](x,t)- \langle Dv^{\eps_{n}}_{\rho_{n},h_{n}}[\mu_{n}](x,t), y-x \rangle \leq C' |y-x|^{2}. \ee
%where 
%$$r_{n}:=  \int_{0}^{1} \langle Dv^{\eps_{n}}_{\rho_{n},h_{n}}[\mu_{n}](x+ \tau (y-x),t)-  Dv^{\eps_{n}}_{\rho_{n},h_{n}}[\mu_{n}](x,t), y-x \rangle \dd \tau.$$
%Therefore, $r_{n}=r_{n1}+r_{n2}$ where\small
%$$ \ba{rcl} r_{n1}&:=& \int_{0}^{\frac{2\rho}{|y-x|}} \langle Dv^{\eps_{n}}_{\rho_{n},h_{n}}[\mu_{n}](x+ \tau (y-x),t)-  Dv^{\eps_{n}}_{\rho_{n},h_{n}}[\mu_{n}](x,t), y-x \rangle \dd \tau,\\[4pt]
%		 r_{n2}&:=& \int_{\frac{2\rho}{|y-x|}}^{1} \langle Dv^{\eps_{n}}_{\rho_{n},h_{n}}[\mu_{n}](x+ \tau (y-x),t)-  Dv^{\eps_{n}}_{\rho_{n},h_{n}}[\mu_{n}](x,t), y-x \rangle \dd \tau.Ê\ea$$\normalsize
%Setting $z(\tau)= x+ \half \tau (y-x)$, Lemma \ref{amadmamdaaaa}  implies the existence of $d_{1}'>0$ (independent of $n$) such that \small
%\be\label{sasmasmamsasas}\ba{rcl} r_{n2}&=& \int_{\frac{2\rho}{|y-x|}}^{1} \langle Dv^{\eps_{n}}_{\rho_{n},h_{n}}[\mu_{n}](z(\tau)+ \frac{\tau}{2} (y-x),t)-  Dv^{\eps_{n}}_{\rho_{n},h_{n}}[\mu_{n}](z(\tau)- \frac{\tau}{2} (y-x),t), y-x \rangle \dd \tau,\\
%\; & \leq & d_{1}' |y-x|^{2} \int_{0}^{1} \tau \dd \tau =  \frac{d_{1}'}{2} |y-x|^{2}.\ea \ee \normalsize
%By \eqref{aadamammsssssaaa} we deduce    
%$$ v^{\eps_{n}}_{\rho_{n},h_{n}}[\mu_{n}](y,t)- v^{\eps_{n}}_{\rho_{n},h_{n}}[\mu_{n}](x,t)- \langle Dv^{\eps_{n}}_{\rho_{n},h_{n}}[\mu_{n}](x,t), y-x \rangle\leq r_{n1}+\frac{d_{1}'}{2} |y-x|^{2}.$$
On the other hand, Lemma \ref{propiedadesdelavmuhrhocontinuaconeps}{\rm(i)}  implies  that $Dv^{\eps_{n}}_{\rho_{n},h_{n}}[\mu_{n}](x,t)$ is uniformly bounded in $n$.
Thus, passing to the limit in the above inequality, every limit point $p$ of $Dv^{\eps_{n}}_{\rho_{n},h_{n}}[\mu_{n}](x,t)$ satisfies
$$ v[\mu](y,t)- v[\mu](y,t)- \langle p, y-x \rangle \leq C' |y-x|^{2} \hspace{0.5cm} \mbox{for all } \hspace{0.2cm} x,y \in \RR^{d}.$$
The above inequality implies that $p \in D^{+}v[\mu](x,t)$ and thus, if $Dv[\mu](x,t)$ exists, the semiconcavity of $v[\mu]$ implies that $p= Dv[\mu](x,t)$ from which the result follows. 
\end{proof}
\subsection{The fully-discrete scheme for the continuity equation}
 Given $\mu \in  C([0,T]; \P_{1})$ and $\eps>0$ let us define 
\be\label{flujofullydiscretounpaso} \Phi^{\eps}_{i,k,k+1}[\mu]:=x_{i}- h \hat{\alpha}^{\eps}_{i,k}[\mu] \hspace{0.4cm} \mbox{for all } i \in \ZZ^{d}, \hspace{0.2cm} k=0,\hdots, N-1, 
\ee
where $\hat{\alpha}^{\eps}_{i,k}:=\hat{\alpha}^{\eps}_{\rho,h}[\mu](x_i, t_k)$ and $\hat{\alpha}^{\eps}_{\rho,h}[\mu]:Ê\RR^{d} \times [0,T] \to \RR^{d}$ is defined as 
\be \label{defalphaeps}
 \hat{\alpha}_{\rho,h}^{\eps}[\mu](x, t):= Dv_{\rho,h}^{\eps}[\mu](x,t). 
\ee

Given the   family  $\{ \Phi_{i,k,k+1}^{\eps}[\mu] \; ; \; i \in \ZZ^{d}, \; k=0, \hdots, N-1\}$,  we now consider  a fully-discrete scheme for  \eqref{eqcontimu} which turns out to be equivalent to the one proposed  \cite{PiccoliTosin}, under some slight change of notation. Let us define  
$$ \mathcal{S}:= \left\{   z= (z_i)_{i\in \Zd} \ ; \ z_{i}\in \RR_{+} \hspace{0.2cm} \mbox{and }    \sum_{i\in \ZZ^d} z_{i} = 1\right\}.$$
 The coordinates of   $m \in \mathcal{S}_{N+1}:=\{ \nu= (\nu_{i})_{k=0}^{N} \; ; \; \nu_{k} \in  \mathcal{S}\}$  are denoted as $m_{i, k}$, with $i \in \ZZ^d$ and $k=0,...,N$.   We set
 $$E_{i}:=  [x_{i}\pm  \half \rho e_{1}] \times ... \times  [x_{i}\pm \half \rho e_{d}] \hspace{0.5cm} \mbox{for all $i\in \ZZ^{d}$},$$
% 
% We identify each  $\mu \in \mathcal{S}_{N+1}$ with the element $\mu \in C([0,T]; \P_{1})$ defined as 
%\be \label{defmisura}
%\mu(x,t):= \frac{1}{\rho^d}\left[\frac{t_{k+1}-t}{h}\sum_{i \in \ZZ^d} \mu_{i, k} \mathbb{I}_{E_i}(x)+\frac{t-t_{k}}{h}\sum_{i \in \ZZ^d} \mu_{i, k+1} \mathbb{I}_{E_i}(x)\right] \; {\text{if }}\; t \in [t_{k}, t_{k+1}]. \ee
%Conversely, every measure $\mu$ of this type can be identified with an element, still written as $\mu$,  in $\mathcal{S}_{ N+1}$ (with coordinates $\mu_{i, k}$). 
and define $m^{\eps}[\mu] \in \mathcal{S}_{N+1}$ recursively as  \small
\be\label{defmconmu}\ba{rcl} m_{i,k+1}^{\eps}[\mu]  &:=& \sum_{j\in \ZZ^{d}} \beta_{i}\left( \Phi^{\eps}_{j,k,k+1}[\mu]\right) m_{j,k}^{\eps}[\mu], \hspace{0.4cm} \mbox{for   } i \in \ZZ^{d}, \hspace{0.2cm}Êk=0,\hdots, N-1,\\[6pt]
						m_{i,0}^{\eps}[\mu]  &:=& \int_{E_{i}} m_{0}(x) \dd x, \hspace{0.4cm} \mbox{for } i \in \ZZ^{d}.\ea\ee \normalsize \smallskip
\begin{remark}
Note that, omitting the dependence in $\mu$, for $k=0, \hdots, N-1$ we have that \small
$$ \sum_{i \in \ZZ^{d}} m_{i,k+1}^{\eps} = \sum_{i\in \ZZ^{d}} \sum_{j \in \ZZ^{d}}  \beta_{i}\left( \Phi^{\eps}_{j,k,k+1} \right) m_{j,k}^{\eps}=  \sum_{j \in \ZZ^{d}} m_{j,k}^{\eps}  \sum_{i\in \ZZ^{d}} \beta_{i}\left( \Phi^{\eps}_{j,k,k+1} \right) =  \sum_{j \in \ZZ^{d}} m_{j,k}^{\eps}=1,   $$
\normalsize
because $\sum_{j\in \ZZ^{d}} m_{j,0}^{\eps}=1.$ Therefore, the scheme \eqref{defmconmu} {\it is conservative}.  \end{remark} \smallskip

 Let us define $m^{\eps}_{\rho,h}[\mu] \in L^{\infty}(\RR^{d}\times [0,T])$     as
\small
\be \label{defmisuradepdemu}
\ba{ll}m_{\rho,h}^{\eps}[\mu](x,t)&:= \frac{1}{\rho^{d}}\left[\frac{t_{k+1}-t}{h}\sum_{i \in \ZZ^d}m_{i,k}^{\eps}[\mu] \mathbb{I}_{E_i}(x)+\frac{t-t_{k}}{h}\sum_{i \in \ZZ^d}m_{i,k+1}^{\eps}[\mu]  \mathbb{I}_{E_i}(x)\right],\\[6pt]
						\ & \hspace{0.4cm} \;  \mbox{if $t\in [t_{k}, t_{k+1})$.} \ea\ee
\normalsize
Therefore, for every $t\in [t_{k}, t_{k+1})$ we have \small
\be \label{defmisuradepdemuotraversion} m_{\rho,h}^{\eps}[\mu](x,t):=  \left(\frac{t_{k+1}-t}{h}\right)m_{\rho,h}^{\eps}[\mu](x,t_{k}) +\left(\frac{t-t_{k}}{h}\right)m_{\rho,h}^{\eps}[\mu](x,t_{k+1}). \ee
\normalsize

%We  now derive a fully discrete scheme for \eqref{RPxCE}, following  the same idea contained in \cite{Camilli}.
% Let us start from a time discrete approximation as formulated in \eqref{escrituraenunpas}. Next let us  project on a space grid by using \eqref{definterpolation} : 
%\be\label{Mfullydiscrete1} \int_{\RR^{d}} I[\phi](x) \dd m_{k+1}[\mu](x)= \int_{\RR^{d}} I[\phi] \left(x- h  \alpha_{k}[\mu](x) \right) \dd m_{k}[\mu](x).\ee
%Next, let us consider a fully discrete approximation of the measure as described in \eqref{defmisura}.
%%if $t \in [t_{k}, t_{k+1}]$
%%\small
%%\be \label{Mfullydiscrete2}
%%m_{\rho,h}^{\eps}[\mu](x,t):= \frac{1}{\rho^{d}}\left[\frac{t_{k+1}-t}{h}\sum_{i \in \ZZ^d}m_{i,k}^{\eps}[\mu] \mathbb{I}_{E_i}(x)+\frac{t-t_{k}}{h}\sum_{i \in \ZZ^d}m_{i,k+1}^{\eps}[\mu]  \mathbb{I}_{E_i}(x)\right] . \ee
%Now, using  definition \eqref{defmisuradepdemu}  together with the definition   \eqref{definterpolation} in \eqref{Mfullydiscrete1}, we  get:
%\be \label{Mfullydiscrete3}
% \sum _{j\in \ZZ^d}\sum _{i\in \ZZ^d}\beta_i(x_j)m_{j,k+1}[\mu] = \sum _{j\in  \ZZ^d} \sum_{i\in \ZZ^d} \beta_{i}\left(x- h  \alpha_{k}(x))[\mu]\right) m_{j,k}[\mu].
%\ee
%Last equality holds if and only if 
%$$m_{i,k+1}[\mu] = \sum _{j\in \ZZ^d} \beta_{i}\left(x- h  \alpha_{k}(x))[\mu]\right) m_{j,k}[\mu]. $$
%To get the convergence of the complete scheme, we need the convergence of the fully discrete flux, which is guaranteed if we consider a regularized discrete flux as described in \eqref{flujofullydiscretounpaso}. 
%%%%%%%%%%%%%%%%%%%%%%%%%%%%%%%%%%%%%%%%%%%%%%%%%%%%%%%%%%%

By abuse of notation, we continue to write   $m^{\eps}_{\rho,h}[\mu](t)$ for the probability measure in $\RR^{d}$ whose density is given by \eqref{defmisuradepdemu}. Thus, by the very definition, we can identify $m_{\rho,h}^{\eps}[\mu](\cdot, \cdot)\in L^{\infty}(\RR^{d}\times [0,T])$ with an element $m^{\eps}_{\rho,h}[\mu](\cdot) \in C([0,T];\P_1)$. 
 
We now study some technical properties of the family  $\{ \Phi_{i,k,k+1}^{\eps}[\mu] \; ; \; i \in \ZZ^{d}, \; k=0, \hdots, N-1\}$. The next result is an easy consequence of Lemma \ref{propiedadesdelavmuhrhocontinuaconeps}. \smallskip
\begin{proposition}{\label{Discretetrajectoriesprop}}For any $i, j \in \ZZ^{d}$ and $k=0,\hdots, N-1$, we have 
\be\label{dis:discretetraj}
|\Phi^{\eps}_{i,k,k+1}[\mu]-\Phi^{\eps}_{j,k,k+1}[\mu]|^{2}\geq \left(1-d_{2} h\left(1+ \frac{\rho^{2}}{\eps^{3}}\right)\right)|x_i-x_j|^2,
\ee
where $ d_2\geq 0$ is   independent of  $(\rho,h,\eps,\mu)$.
\end{proposition}
\begin{proof}
%In the following, we will denote by $C$ potive constants depending on $h,\rho$
For the reader's convenience, we omit the $\mu$ argument. Recalling \eqref{flujofullydiscretounpaso} and \eqref{defalphaeps}, for every $k=0,\hdots, N-1$ we  have
$$
\ba{rcl} |\Phi^{\eps}_{i,k,k+1}-\Phi^{\eps}_{j,k,k+1}|^2&=& \left|x_i-x_j -h \left[D v_{\rho,h}^{\eps} (x_i,t_{k})-D v_{\rho,h}^{\eps}(x_j,t_{k})\right]\right|^2,\\[6pt]
										\      	&=&  |x_i-x_j |^2 +h^2 |D v_{\rho,h}^{\eps}(x_i,t_k)-D v_{\rho,h}^{\eps}(x_j,t_k)|^2+\\[4pt]
										\      	& \ & -2h \langle D v_{\rho,h}^{\eps} (x_i,t_k)-D v_{\rho,h}^{\eps} (x_j,t_k),x_i-x_j \rangle,
\ea
$$
which yields to 
$$ |\Phi^{\eps}_{i,k,k+1}-\Phi^{\eps}_{j,k,k+1}|^2 \geq |x_i-x_j |^2-2h \langle D v_{\rho,h}^{\eps} (x_i,t_k)-D v_{\rho,h}^{\eps} (x_j,t_k),x_i-x_j \rangle.$$
Therefore, by Lemma  \ref{propiedadesdelavmuhrhocontinuaconeps}{\rm(ii)}, there exists $d_{2}>0$ such that
\eqref{dis:discretetraj} holds. 
%$$ |\Phi^{\eps}_{i,k,k+1}-\Phi^{\eps}_{j,k,k+1}|^2 \geq  (1-Ch) |x_i-x_j |^2 -Ch\rho^{2},$$
%\begin{eqnarray}
%&&|\Phi^{\rho,h}(x_i,t_n;h)-\Phi^{\rho,h}(x_j,t_n;h)|^2=|x_i-x_j -h(\hat{\alpha}^k_{\eps,i}-\hat{\alpha}^k_{\eps,j})|^2\geq \nonumber\\
%&&|x_i-x_j |^2-2h(D v_{\rho,h}^{\eps}[m^{\rho,h}](x_i,t_n)-D v_{\rho,h}^{\eps}[m^{\rho,h}](x_j,t_n),x_i-x_j) \label{dis1}
%\end{eqnarray}
%By the weak semiconcavity property of$v_{\rho,h}^{\eps}[m^{\rho,h}]$, the following inequality holds
%\be \label{stimaperDv_hro}
%(D v_{\rho,h}^{\eps}[m^{\rho,h}](x_i,t_n)-D v_{\rho,h}^{\eps}[m^{\rho,h}](x_i,t_n),x_i-x_j)\leq C_1(|x_i-x_j|^2+C_2 d \rho^2
%\ee
%where here $C_1$ is denoting the semi concavity constant of $u^\eps$ not depending on $h\rho$ and $C_2$ is the constant in the interpolation error of $|x^2|.$\\
%We use \eqref{stimaperDv_hro} in \eqref{dis1} and we get:
%\begin{eqnarray*}
%&&|\Phi^{\rho,h}(x_i,t_n;h)-\Phi^{\rho,h}(x_j,t_n;h)|^2=|x_i-x_j -h(\hat{\alpha}^k_{\eps,i}-\hat{\alpha}^k_{\eps,j})|^2\geq \nonumber\\
%&&|x_i-x_j |^2-2h(C_1|x_i-x_j|^2+C_2 d \rho^2)
%\end{eqnarray*}
%which finishes the proof.
\end{proof}\medskip

Now we provide a technical result which,  in the case $d=1$,  allow us  to obtain uniform  $L^{\infty}$ bounds for  $m^{\eps}_{\rho,h}[\mu]$ (see Proposition \ref{Lemmaboundform_hro}{\rm(ii)} below).   \smallskip
\begin{lemma}\label{resultadotecnico} Suppose that $d=1$ and that $\rho^{2}/\eps^{3} \leq d_{2}'$, with $d_{2}'>0$ (independent of $(\rho,h,\eps,\mu)$). Then, there exists a constant $d_{3}>0$ (independent of $h$ small enough and $(\rho,\eps,\mu)$) such that for any   $i\in \ZZ$ and $k=0,\hdots,N-1$, we have that 
\be\label{cotasdelnumero}\sum_{j\in\ZZ} \beta_{i}\left( \Phi^{\eps}_{j,k,k+1} \right)\leq 1+ d_{3}h. \ee
\end{lemma}

\begin{proof} For notational simplicity, let us set $y_{j}=  \Phi^{\eps}_{j,k,k+1}$. Note that for any $j_{1}, j_{2} \in \ZZ$, Proposition \ref{Discretetrajectoriesprop} implies that 
$$ \left|y_{j_{1}} - y_{j_{2}} \right|^{2} \geq \left(1-d_{3}'h\right) \left| x_{j_{1}} - x_{j_{2}} \right|^{2} ,$$
where $d_{3}'= d_{2} (1+ d_{2}')$.  Thus, if $j_{1}\neq j_{2}$, we get 
\be\label{estimaciraiz} \left|y_{j_{1}} - y_{j_{2}} \right|^{2} \geq \left(1-d_{3}'h\right) \rho^{2}, \quad \mbox{i.e. } \left|y_{j_{1}} - y_{j_{2}} \right|  \geq \sqrt{\left(1-d_{3}'h\right)} \rho.\ee
Since the diameter of   $\mbox{supp}(\beta_{i})$ is equal to $2\rho$, the above inequality implies that for $h$ small enough (independent of $(\rho,\eps,\mu)$), the cardinality of 
$$ \Z_i:=\left\{ j \in \ZZ \; ; \;  y_{j} \in \mbox{supp}(\beta_{i})) \right\}$$
is at most $3$.  If   $\Z_i$ only has one element, then \eqref{cotasdelnumero} is trivial. If $\Z_i$  has two elements $y_{j_{1}}$, $y_{j_{2}}$ with $y_{j_{1}} < y_{j_{2}}$, then
$$ \beta_{i}(y_{j_{1}}) +  \beta_{i}(y_{j_{2}})= 2- \frac{|y_{j_{1}} - x_{i}|}{\rho}- \frac{|y_{j_{2}} - x_{i}|}{\rho} \leq 2-  \frac{|y_{j_{1}} - y_{j_{2}}|}{\rho},$$
by the triangular inequality. Using \eqref{estimaciraiz} we get
$$ \beta_{i}(y_{j_{1}}) +  \beta_{i}(y_{j_{2}}) \leq 2 - \sqrt{\left(1-d_3'h\right)} \leq 1+ d_3'h, $$
from which \eqref{estimaciraiz} follows. Finally, if  $\Z_i$ has three elements $y_{j_{1}}$, $y_{j_{2}}$ and $y_{j_{3}}$, then (supposing for example that $y_{j_{1}} \leq y_{j_{2}} \leq x_{i} < y_{j_{3}}$) we have 
$$\ba{ll} \beta_{i}(y_{j_{1}}) +  \beta_{i}(y_{j_{3}}) & =1- \frac{x_{i}-y_{j_{1}}}{\rho} + 1 - \frac{y_{j_{3}}-x_{i}}{\rho},\\
									\    &= 2-\frac{y_{j_{2}}-y_{j_{1}}}{\rho}   - \frac{y_{j_{3}}-y_{j_{2}}}{\rho} \leq 2 - 2\sqrt{\left(1-d_3'h\right)} \leq 2d_3'h. \ea$$
Using that  $\beta_{i}(y_{j_{2}}) \leq 1$ and the above  estimate, we obtain \eqref{estimaciraiz} with $d_3:= 2d_3'$. \end{proof}\medskip

%\begin{remark} It is easy to see that for $d=2$, inequality \eqref{estimaciraiz} is useless in order to prove    an estimate  like \eqref{cotasdelnumero}. In fact, using the notations of the above proof, supposing for example that $x_{i}= 0$  and setting $\gamma:=\sqrt{\left(1-2d_{2}h\right)}$,  inequality \eqref{estimaciraiz} does not exclude the existence of $j_1$, $j_2$, $j_3\in \ZZ^{2}$ such that 
%$$y_{j_{1}}= (- \half \gamma \rho, 0), \;  \; y_{j_{2}}= ( \half \gamma \rho,0) \; \;  \mbox{and } \; y_{j_{3}}= \left(0, \frac{\sqrt{3}}{2} \gamma \rho\right).$$
%In this  case, we have that
%$$ \beta_{i}(y_{j_{1}})+ \beta_{i}(y_{j_{2}})+\beta_{i}(y_{j_{3}})= 2 (1- \half \gamma) + 1-  \frac{\sqrt{3}}{2} \gamma= 2-  \frac{\sqrt{3}}{2} + O(h) .$$
%\end{remark}\smallskip
Using the above results, we can establish some important properties for $m^{\eps}_{\rho, h}[\mu]$, which are similar to those in Theorem \ref{expresionentermtransp}.\smallskip

\begin{proposition}\label{Lemmaboundform_hro} Suppose that $\rho=O(h)$. Then, there exists  a constant  $d_{4} >0$ (independent of  $(\rho,h,\eps,\mu)$) such that:\smallskip\\
{\rm(i)} For all $t_{1},t_{2} \in [0,T]$, we have that
\be \db_{1}(m^{\eps}_{\rho,h}[\mu](t_{1}),m^{\eps}_{\rho,h}[\mu](t_{2}) )\leq d_4|t_{1}-t_{2}|.\ee
{\rm(ii)}   For all $t\in [0,T]$,  $m^{\eps}_{\rho,h}[\mu](t)$    has a support in $B(0,d_4)$.\smallskip\\
 {\rm(iii)} If    $d=1$ and $\rho=O(\eps^{3/2})$, then we have $$\|m^{\eps}_{\rho,h}[\mu](\cdot,t)\|_{\infty}\leq d_{4}.$$
\end{proposition}
\begin{proof} Let  $\phi \in C(\RR^{d})$ be  a $1$-Lipschitz function. By   \eqref{defmisuradepdemuotraversion}, the function $\psi_{\phi}: [0, T] \to \RR$, defined as 
$$\psi_{\phi}(t):=  \int_{\RR^{d}} \phi(x) \dd m^{\eps}_{\rho,h}[\mu](t),$$
is   affine   in each interval $[t_{k}, t_{k+1}]$, with $k=0,\hdots, N-1$.  It clearly belongs to $W^{1,\infty}([0,T])$ and
$$ \left\| \frac{d}{dt} \psi_{\phi} \right\|_{\infty}= \frac{1}{h}\max_{k=0,\hdots, N-1} \left| \int_{\RR^{d}} \phi(x) \dd [m^{\eps}_{\rho,h}[\mu](t_{k+1})-m^{\eps}_{\rho,h}[\mu](t_{k})]\right|.$$
For every $k=0,\hdots, N-1$ we have, omitting   $\mu$ from the notation, \footnotesize
\begin{eqnarray*} \int_{\RR^{d}} \phi(x) \dd [m^{\eps}_{\rho,h}(t_{k+1})-m^{\eps}_{\rho,h}(t_{k})]=  \frac{1}{\rho^{d}}\underset{i\in \ZZ^{d}} \sum\int_{E_{i}} \phi(x) \dd x \left[\underset{j\in \ZZ^{d}}\sum \beta_{i}\left( \Phi^{\eps}_{j,k,k+1}\right) m^{\eps}_{j,k}-m_{i,k}^{\eps}    \right],&\\
=  \underset{j\in \ZZ^{d}} \sum m^{\eps}_{j,k}  \left[\underset{i\in \ZZ^{d}}\sum \beta_{i}\left( \Phi^{\eps}_{j,k,k+1}\right) \frac{1}{\rho^{d}}\int_{E_{i}} \phi(x) \dd x - \frac{1}{\rho^{d}}\int_{E_{j}} \phi(x) \dd x\right]. &  \end{eqnarray*}
\normalsize
On the other hand, since  $\phi$ is $1$-Lipschitz, we have that
\be\label{cambiointegralpunto} \left|\frac{1}{\rho^{d}}Ê\int_{E_{i}} \phi(x) \dd x - \ \phi(x_{i}) \right| \leq  \rho. \ee
Using \eqref{cambiointegralpunto}, estimate  \eqref{www}, Lemma \ref{propiedadesdelavmuhrhocontinuaconeps}{\rm(i)} and the fact that $\rho=O(h)$, we get that \footnotesize	
$$\ba{ll} \left|\int_{\RR^{d}} \phi(x) \dd [m^{\eps}_{\rho,h}(t_{k+1})-m^{\eps}_{\rho,h}(t_{k})]\right| &\leq	 \underset{j\in \ZZ^{d}}\sum m^{\eps}_{k,j}  \left| \underset{i\in \ZZ^{d}}\sum  \beta_{i}\left( \Phi^{\eps}_{j,k,k+1}\right)   \phi(x_{i})  -   \phi(x_{j}) \right|+ 2 \rho,\\[6pt]
		\ & =  \underset{j\in \ZZ^{d}}\sum 	m^{\eps}_{k,j}  \left| \phi \left( \Phi^{\eps}_{j,k,k+1} \right) - \phi(x_{j})\right|+ 2 c \rho, \\[6pt]
		\ &  \leq  d_{0} h + 2c \rho = \left(  d_{0} + \frac{2c\rho}{h} \right) h\leq c'h,\ea$$
\normalsize
for some constants  $c$, $c'  >0$     independents of  $(\rho,h,\eps,\mu)$.  Therefore, we obtain that $ \left\| \frac{d}{dt} \psi_{\phi} \right\|_{\infty}\leq c'$, which proves {\rm(i)} with $d_4$ to be chosen later.

In order to prove {\rm(ii)}, it suffices to note that since $\|Dv^{\eps}_{\rho,h}[\mu]\|_{\infty} \leq d_0$  we easily check that $\supp(m^{\eps}_{\rho,h}[\mu](t))\subset B(0,c_1+ 2d_{0} T)$.
Now, let us assume $d=1$.  By the definition of $m^{\eps}_{\rho,h}[\mu](\cdot,0)$ in \eqref{defmisuradepdemu} and assumption {\bf (H1)}, we have
$$\|m^{\eps}_{\rho,h}[\mu](\cdot,0)\|_{\infty}= \max_{i\in \ZZ} \left\{ \frac{1}{\rho }m^{\eps}_{i,0}[\mu] \right\} \leq  \|m_{0}\|_{\infty}\leq c_1.$$
Now, given $k=0,\hdots, N-1$, we have that 
$$ \| m_{\rho,h}^{\eps}[\mu](\cdot,t_{k+1}) \|_{\infty} \leq \max_{i\in \ZZ } \left\{ \frac{1}{\rho} m_{i,k+1}^{\eps}[\mu] \right\}=  \frac{1}{\rho} \max_{i\in \ZZ }\left\{\sum_{j\in \ZZ} \beta_{i}\left( \Phi^{\eps}_{j,k,k+1}[\mu] \right) m_{j,k}^{\eps}[\mu]\right\}. $$
Therefore, by Lemma \ref{resultadotecnico}, we obtain that  \small
$$\| m_{\rho,h}^{\eps}[\mu](\cdot,t_{k+1}) \|_{\infty} \leq  \| m_{\rho,h}^{\eps}[\mu](\cdot,t_{k}) \|_{\infty} \sum_{j\in\ZZ }  \beta_{i}\left( \Phi^{\eps}_{j,k,k+1}[\mu] \right)\leq (1+ d_{3}h) \| m_{\rho,h}^{\eps}[\mu](\cdot,t_{k}) \|_{\infty}. $$\normalsize
Iterating in the above expression, we obtain that 
$$\| m_{\rho,h}^{\eps}[\mu](\cdot,t_{k+1}) \|_{\infty} \leq  (1+ d_{3}h)^{\frac{T}{h}} \| m_{0} \|_{\infty}\leq e^{d_{3}T}c_{1},$$
for $h$ small enough. The result follows, taking $d_{4}= \max\{ c',c_1+ 2d_{0} T, e^{d_{3}T}c_{1}\}$.\end{proof}

\subsection{The fully-discrete scheme for the first order MFG problem \eqref{MFG}}
%%%%%%%%%%%%CONVERGENZA FULLY DISCRETE
For a given $\rho,h, \eps>0$ and $\mu \in \mathcal{S}^{N+1}$ we still write $\mu$ for the element in $C([0,T]; \P_{1})$ defined as 
\be \label{defmisura}
\mu(x,t):= \frac{1}{\rho^d}\left[\frac{t_{k+1}-t}{h}\sum_{i \in \ZZ^d} \mu_{i, k} \mathbb{I}_{E_i}(x)+\frac{t-t_{k}}{h}\sum_{i \in \ZZ^d} \mu_{i, k+1} \mathbb{I}_{E_i}(x)\right] \; {\text{if }}\; t \in [t_{k}, t_{k+1}]. \ee
Let us  consider the following {\it full  discretization} of (MFG):  
%\begin{equation}\label{MFGh}
%\left\{
%\begin{array}{rl}
%v_h(x,t_{n})&=\inf_{\alpha \in \RR^{d}}\left\{v_h(x-h\alpha,t_{n+1})+\half h | \alpha|^{2} \right\}+hF(x,m_h(t_{n})); \ x\in \RR^{d},   k=0,\dots,N-1, \\[4pt]
% m_h(k)&=\Phi_{h}(\cdot,k)\sharp m_{0}  \hspace{0.4cm} Ê\mbox{for } k=1,\dots,N, \\[4pt]
%m_h(0)&=m_0 \in \P_{1}, \hspace{0.3cm} v_h(x,N)=G(x,m_h(N))  \hspace{0.4cm} Ê\mbox{for } x \in \RR^{d}.
%\end{array}
%\right.
%\end{equation}
\be\label{MFGfully} \mbox{Find $\mu \in \mathcal{S}_{N+1}$ such that } \hspace{0.2cm}  \mu_{i,k}=m_{i,k}^{\eps}[\mu]\hspace{0.5cm} \mbox{$\forall \; i\in \ZZ^{d}$ and } k=0,\hdots, N,
\ee
where we recall that $m_{i,k}^{\eps}[\mu]$ is defined in \eqref{defmconmu}.
In order to prove that \eqref{MFGfully} admits at least a solution, we will need   the following stability result.\smallskip
\begin{lemma}\label{mepreguntosibuenaestab} Let $\mu^{n} \in \mathcal{S}_{N+1}$ be   a sequence converging to $\mu \in \mathcal{S}_{N+1}$. Then:\smallskip\\
{\rm(i)} $v_{\rho, h}^{\eps}[\mu^{n}](\cdot,\cdot) \to v^{\eps}_{\rho,h}[\mu](\cdot,\cdot)$ uniformly over compact sets. \smallskip\\
{\rm(ii)} $m^{\eps}_{i,k}[\mu^{n}] \to m^{\eps}_{i,k}[\mu]$  for all $i\in \ZZ^{d}$ and  $k=0,\hdots, N$.
\end{lemma}
\begin{proof}
Because of the assumptions on $F$ and $G$ in {\bf (H1)} we clearly have {\rm(i)}.  By definition of $v^{\eps}_{\rho,h}[\mu^{n}](x,t)$ and {\rm(i)}, Lebesgue theorem implies that  we have pointwise convergence of $ Dv^{\eps}_{\rho,h}[\mu^{n}]$  to $Dv^{\eps}_{\rho,h}[\mu]$  and obviously also of  $\hat{\alpha}^{\rho,h}_{\eps}[\mu^{n}](\cdot,\cdot)\to \hat{\alpha}^{\rho,h}_{\eps}[\mu](\cdot,\cdot)$. Assertion {\rm(ii)} for  $i\in \ZZ^{d}$ and $k=1$ follows hence from  the definition \eqref{defmconmu} of  $m^{\eps}_{i,1}[\mu^{n}]$. Therefore, by recursive argument we get the result for all $i\in \ZZ^{d}$ and $k=0,\hdots, N-1$.
\end{proof}\smallskip
\begin{theorem}\label{existenciadealmenosuno} There exists at least one solution of \eqref{MFGfully}.
\end{theorem}
\begin{proof} This is a straightforward consequence of Lemma \ref{mepreguntosibuenaestab}, Proposition \ref{Lemmaboundform_hro}{\rm(ii)}    and Brouwer fixed-point theorem.
\end{proof}\smallskip

Given a solution $m^{\eps}\inÊ\mathcal{S}_{N+1}$ of \eqref{MFGfully},  we set $m^{\eps}_{\rho,h}(\cdot,\cdot)$ for the extension to $\RR^{d}\times [0,T]$ defined in \eqref{defmisuradepdemu}.\smallskip

Now we prove our main result. \smallskip

\begin{theorem}\label{resultadoprincipalfullydiscrete} Suppose that  $d=1$ and that {\bf (H1)-(H3)} hold.  Consider a sequence of positive numbers $\rho_{n}, h_{n}, \eps_{n}$ satisfying that $\rho_{n}= o\left(h_{n} \right)$, $h_{n}= o(\eps_{n})$ and   $\rho_{n}=\ds O(\eps_{n}^{3/2})$ as $\eps_n \downarrow 0$. Let    $\{m^{n}\}_{n\in \mathbb{N}}$ be a sequence of solutions of   \eqref{MFGfully} for the corresponding parameters $\rho_{n}, h_{n}, \eps_{n}$. Then every limit point  in $C([0,T];\P_{1})$ of $m^{n}$ (there exists at least one) solves (MFG). In particular, if {\bf (H4)} holds we have that $m_{\rho_{n},h_{n}}^{\eps_{n}}\to m$ (the unique solution of  (MFG)) in  $C([0,T];\P_{1})$  and in  $L^{\infty}\left(\RR^{d}\times [0,T] \right)$-weak-$\ast$.
\end{theorem}
\begin{remark}
The assumption $\rho=o(h)$  has the form of an inverse CFL condition
%which generally arise in the stability analysis of explicit finite difference schemes  for advection equation. 
 and is typical  for Semi-Lagrangian schemes  (see e.g. \cite{falconeferretilibro}), which   allow large time steps. 
 \end{remark}
\begin{proof} For notational convenience we will write  $v^{n}:= v_{\rho_{n},h_{n}}^{\eps_{n}}[m^{n}]$. By  Proposition \ref{Lemmaboundform_hro}{\rm(i)} and Ascoli theorem we can assume the existence of $\ov{m}\in C([0,T]; \P_{1})$ such that   $m^{n}$ (as an element of $C([0,T]; \P_{1})$) converge to $\ov{m}$ in $ C([0,T]; \P_{1})$.    Moreover,  Proposition \ref{Lemmaboundform_hro}{\rm(iii)} implies that, up to some subsequence,   $m^{n}$ (as an element of $L^{\infty}(\RR^{d}\times [0,T])$) converge in   $L^{\infty}\left(\RR^{d}\times [0,T] \right)$-weak-$\ast$ to some $\hat{m}$. Thus, we necessarily have that  $\ov{m}$ is absolutely continuous and its density, still denoted as $\bar{m}$,  is equal to  $\hat{m}$.   In order to complete the proof, we now   show   that $\ov{m}$ solves the continuity equation \eqref{defsoldistribucionsoloenespacio}, i.e. for any $t\in [0,T]$ and $\phi \in C_{c}^{\infty}(\RR^{d})$ 
\be\label{eqfinal}
 \int_{\RR} \phi(x) \dd \ov{m}(t)(x)=  \int_{\RR} \phi(x) \dd m_{0}(x) -  \int_{0}^{t}\int_{\RR} D \phi(x) D v [\ov{m}](x,s) \dd \ov{m} (s)(x)\dd s.\ee
Given $t\in [0,T]$, let us set $t_{n}:=  \left[ \frac{t}{h_{n}} \right] h_{n}$. We have
\be\label{teofinec1} \int_{\RR} \phi(x) \dd m^{n}(t_{n}) = \int_{\RR} \phi(x) \dd m_{0}(x)+   \sum_{k=0}^{n-1}  \int_{\RR}\phi(x) \dd  \left[m^{n}(t_{k+1})-m^{n}(t_{k})\right].\ee
By definitions \eqref{defmconmu} and \eqref{defmisuradepdemu}, setting $\Phi^{n}_{i,k,k+1}:=x_{i}- h_n  D v^{n}(x_i,t_k)$,  for all $k=0,\hdots, n-1$ we have 
\be\label{serieecsdemfinal}\ba{rcl}  \int_{\RR}\phi(x) \dd m^{n}(t_{k+1}) &=&\sum_{i\in  \ZZ} m^{n}_{i,k+1}  \frac{1}{\rho_{n} }  \int_{E_{i}} \phi(x) \dd x,\\[6pt]
									\     & =&  \sum_{i\in \ZZ} \frac{1}{\rho_{n} }  \int_{E_{i}} \phi(x) \dd x \sum_{j\in \ZZ} \beta_{i}\left(\Phi_{j,k,k+1}^{n}   \right)m^{n}_{j,k}, \\[6pt]
									\ 	&=&  \sum_{j\in \ZZ} m^{n}_{j,k} \sum_{i\in \ZZ} \beta_{i}\left(\Phi_{j,k,k+1}^{n} \right) \frac{1}{\rho_{n} }  \int_{E_{i}} \phi(x) \dd x. \ea \ee
As in \eqref{cambiointegralpunto} we get 
$$ \left|\frac{1}{\rho_n}Ê\int_{E_{i}} \phi(x) \dd x - \ \phi(x_{i}) \right| \leq  \| D\phi \|_{\infty} \rho_n.$$
Therefore, combining with \eqref{serieecsdemfinal}, we get (recalling  \eqref{www} with $\gamma=1$)
\be\label{seciguald}\ba{rcl} \int_{\RR}\phi(x) \dd m^{n}(t_{k+1})& =& \sum_{j\in \ZZ} m^{n}_{j,k} \sum_{i\in \ZZ} \beta_{i}\left(\Phi_{j,k,k+1}^{n} \right)\phi(x_{i}) + O(\rho_n), \\[6pt]
									\ &= &   \sum_{j\in \ZZ} m^{n}_{j,k}  I [\phi] \left(\Phi_{j,k,k+1}^{n}\right) + O(\rho_n), \\[6pt]
									\ &= &   \sum_{j\in \ZZ} m^{n}_{j,k}  \phi\left(\Phi_{j,k,k+1}^{n} \right) + O(\rho_n).\ea \ee
%Analogously, 
%$$  \int_{\RR}\phi(x) \dd m^{n}(t_{k})=\sum_{j\in \ZZ^{d}} m^{n}_{j,k}  \phi(x_{j}) + O(\rho).$$
On the other hand, by Lemma \ref{propiedadesdelavmuhrhocontinua}{\rm(i)}, the function  $v^{n}(\cdot, t)$ is Lipschitz (with Lipschitz constant independent of $n$). Therefore, by \eqref{aproximacionuniformedelaconvolucion}   we have the existence of a constant $c>0$ (independent of $n$) such that 
\be\label{stimaDv} \left| D v^{n}(x,t) - D v^{n}(y,t) \right| \leq \frac{c}{\eps_{n}} |x-y|,\ee
which implies, setting $\Phi^{n}_{k,k+1}(x)=x- h_n D v^{n}(x,t)$, that
$$ \left| \phi\left(\Phi^{n}_{k,k+1}(x) \right)-\phi\left(\Phi^{n}_{k,k+1}(y) \right)\right| \leq   c' \left(1+\frac{h_n}{\eps_{n}}\right) |x-y|.$$
for some $c'>c$ (which is also independent of $n$).  Therefore, we have 
$$ \left|\frac{1}{\rho_n}Ê\int_{E_{j}}  \phi\left(\Phi^{n}_{k,k+1}(x)\right) \dd x - \phi\left(\Phi_{j,k,k+1}^{n}\right) \right| \leq c' \left(1+\frac{h_n}{\eps_{n}}\right) \rho_n.$$
Since $\frac{h_n}{\eps_{n}} =O(1)$, by \eqref{seciguald}, we get 
$$\ba{rcl} \int_{\RR}\phi(x) \dd m^{n}(t_{k+1})&=&  \sum_{j\in \ZZ} m^{n}_{j,k} \frac{1}{\rho_{n}}\int_{E_{j}} \phi\left(\Phi^{n}_{k,k+1}(x) \right) \dd x +  O\left(\rho_n\right), \\
									\  &= & \int_{\RR}  \phi\left(\Phi^{n}_{k,k+1}(x) \right) \dd  m^{n}(t_{k}) +  O\left(\rho_n\right).\ea $$
The expression above yields to
\be\label{ecsda}\ba{rcl} \int_{\RR}\phi(x) \dd  \left[m^{n}(t_{k+1})-m^{n}(t_{k})\right] &=&\int_{\RR}\left[   \phi\left(\Phi^{n}_{k,k+1}(x) \right) - \phi(x)\right] \dd  m^{n}(t_{k})   + O\left(\rho_n\right),\\[4pt]
 												              \    &= &- h_{n}\int_{\RR} D\phi(x)Dv^{n}(x,t_{k}) \dd  m^{n}(t_{k}) \\[4pt]
												              \     & \ &  + O\left(h_{n}^{2}+\rho_{n}\right)\ea\ee
 Since  $D\phi(\cdot)\cdot Dv^{n}(\cdot ,t_{k})$ is $c''/\eps_{n}$-Lipschitz (with $c''$ large enough), Proposition  \ref{Lemmaboundform_hro}{\rm(i)} gives that for all $s\in [t_{k}, t_{k+1}]$, with $k=0,\hdots, n-1$, we have
 $$ \left|\int_{\RR}  D\phi(x)Dv^{n}(x,t_{k}) \dd \left[ m^{n}(s)-  m^{n}(t_{k})\right]\right| \leq \frac{c''}{\eps_{n}} |s-t_k| \leq \frac{c''h_{n}}{\eps_{n}},$$
 which implies that, using that $Dv^{n}(x,s)= Dv^{n}(x,t_k)$ for $s\in [t_k, t_{k+1}[$,  
 \be\label{unaexp} \left| \int_{t_{k}}^{t_{k+1}}\int_{\RR}  D\phi(x) Dv^{n}(x,s)  \dd \left[ m^{n}(s)-  m^{n}(t_{k})\right]\dd s \right|\leq \frac{c''h_{n}^{2}}{\eps_{n}}.\ee 
 Therefore, combining \eqref{unaexp} and \eqref{ecsda}, we obtain that 
 $$\ba{rcl}  \int_{\RR}\phi(x) \dd  \left[m^{n}(t_{k+1})-m^{n}(t_{k})\right]  &=& -\int_{t_{k}}^{t_{k+1}}\int_{\RR}  D\phi(x) Dv^{n}(x,s)  \dd  m^{n}(s)(x) \dd s\\[4pt]
 											       \ & Ê\ & +O\left( \frac{h_{n}^{2}}{\eps_{n}}+ \rho_{n}\right).\ea $$
 Thus, summing from $k=0$ to $k=n-1$ and using  \eqref{teofinec1}
 \be\label{antesdepasarallimite}\ba{rcl}  \int_{\RR} \phi(x) \dd m^{n}(t_{n})(x)&=& \int_{\RR} \phi(x)  m^{n}(x,0)-\int_{0}^{t_{n}}\int_{\RR}  D\phi(x) Dv^{n}(x,s)  m^{n}(x,s)\dd x \ \dd s \\[6pt]
 									\       & \ & + O\left( \frac{h_{n}}{\eps_{n}}+\frac{\rho_{n}}{h_{n}}\right). \ea\ee
By Theorem \ref{covergencefullydiscreteconeps} we have that  $Dv^{n}(x,s) \to Dv[\bar{m}](x,s)$ for a.a. $(x,s) \in \RR\times [0,T]$. Therefore, using that $\phi\in C^{\infty}_{c}(\RR)$,   the Lebesgue theorem implies that
$$  \mathbb{I}_{[0,t_{n}]} D\phi(\cdot)\cdot Dv^{n}(\cdot,\cdot)  \rar   \mathbb{I}_{[0,t]} D\phi(\cdot)\cdot Dv[\bar{m}](\cdot,\cdot) \in L^{1}( \RR \times [0,T]) \hspace{0.3cm} \mbox{strongly in $L^1$},$$
and since $m^{n}$ converge to $\bar{m}$ in $L^{\infty}\left(\RR\times [0,T] \right)$-weak-$\ast$, we can pass to the limit in \eqref{antesdepasarallimite} to obtain \eqref{eqfinal}. The result follows.
%$$ m^{n}(x,s)-  m^{n}(x,t_{k})= \frac{s-t_{k}}{h} \left(  m^{n}(x,t_{k+1})-m^{n}(x,t_{k+1})   \right),$$ 
%expression \eqref{unaexp} can be estimated by 
% $$ \int_{t_{k}}^{t_{k+1}}  \frac{(s-t_{k})}{h} \left|\int_{\RR}  D\phi(x)\cdot Dv^{n}(x,t_{k})  \left[ m^{n}(x,t_{k+1})-  m^{n}(x,t_{k})\right] \right| \dd s.$$ 
% Since    $D\phi(\cdot)\cdot Dv^{n}(\cdot,t_{k})$ is $C/\eps$-Lipschitz (with $C$ big enough), Lemma ? gives
%												              
%Since $\|D\phi(\cdot)Dv^{n}(\cdot,t_{n})\|_{\infty}$ is uniformly bounded, we have that 			
%  $$ \int_{\RR}\phi(x) \dd  \left[m^{n}(t_{k+1})-m^{n}(t_{k})\right]= 	- h_{n}\int_{\RR} D\phi(x)Dv^{n}(x,t_{n}) \dd  m^{n}(t_{k}) + O\left(d_{1}(m^{n},\bar{m})+ h_{n}^{2}+\frac{\rho_{n}}{\eps_{n}}\right).$$
%On the other hand, since
% $$  		- h_{n}\int_{\RR} D\phi(x)Dv^{n}(x,t_{n}) \dd  m^{n}(t_{k})$$			           
\end{proof}
\section{Numerical Tests}\label{numericaltests}
We show numerical simulations for the case $d=1$.  Given $\eps$, $\rho$, $h>0$ we set $\{m_{i,k}^{\eps} \; ; \; i \in \ZZ^{d}, \; k=0,\hdots, \left[\frac{T}{h}\right]\}$ for the solution  of \eqref{MFGfully} and  $\{v_{i,k}^{\eps} \; ; \; i \in \ZZ^{d}, \; k=0,\hdots, \left[\frac{T}{h}\right]\}$ for the associate value functions.  We approximate heuristically  $m_{i,k}^{\eps}$ and $v_{i,k}^{\eps}$   with    a fixed--point iteration method. We consider as initial   guess the element in $m^{\eps,0} \in \mathcal{S}_{N+1}$ given by
 %Omitting $\mu$ to simplify notations,  given 
  $$m^{\eps,0}_{i,k}=m^{\eps}_{i,0}=\int_{E_i}m_0(x)\dd x,\quad i\in \mathbb{Z},\; k=0,\dots,N.$$
Next, for $p=0,1,2, \hdots$, given $m^{\eps,p} \in \mathcal{S}_{N+1}$ we calculate  $v^{\eps,p+1}\in B(\G_{\rho,h})$ with the backward scheme \eqref{scheme-control}, taking as $\mu$ the extension of $m^{\eps,p}$ to $C([0,T]; \P_{1})$ defined in \eqref{defmisura}. The element $ m^{\eps,p+1} \in \mathcal{S}_{N+1}$  is then computed with the forward scheme \eqref{defmconmu}, taking 
%
%
%
% let us compute  for $p=1,2,...$, $i\in \mathbb{Z}^d ,$ and $ k=0,\dots{N-1}$,
%\be\label{FP}
% \begin{cases} 
%  %v^{p+1}_{i,k}=S_{\rho, h}[m^{\eps, p}](v^{p+1}_{\cdot,k+1},i,k) & i\in \mathbb{Z}^d\\
%  v^{p+1}_{i,k}=  \inf_{\alpha\in\RR}\{I[v^{p+1}_{\cdot,k+1}] (x_i-h \alpha )+\frac{1}{2}h \alpha^2 )\}+hF(x_i,m^{\eps,p}_{\rho,h}(t_k))\\% &i\in \mathbb{Z},\; k=N-1,\dots,0\\
%  v^{p+1}_{i,N}=G(x_i,m^{\eps,p}_{\rho,h}(T))\\%& i\in \mathbb{Z}\\ 
%%  \hat{\alpha}^{\eps}_{i,k}:= Dv_{i,k}^{\eps,p+1}\\ 
%  m^{\eps,p+1}_{i,k+1}=\sum_{j}  \beta_i( x_j-h \hat{\alpha}^{\eps,p+1}_{j,k})m^{\eps,p+1}_{j,k}  \\%&i\in \mathbb{Z}^d ,\; k=0,\dots{N-1}\\
%  m^{\eps,p}_{i,0}=m^{\eps}_{i,0}.% &i\in \mathbb{Z}^d \\%  
% \end{cases}
%\ee
%%where as initial guess for the fixed point we choose $m^{\eps,0}_{i,k}=\int_{E_i}m_0(x)\dd x $.\\
%%  $Dv_{i,k}^{\eps,p+1})$ we denote an approximation of  $\hat{\alpha}^{\eps}_{i,k}$ computed as following.
%Where the values $\hat{\alpha}^{\eps,p}_{j,k}$  represent  the  gradients
%$$ \hat{\alpha}^{\eps,p}_{j,k}=D v_{\rho,h}^{\eps,p}(x_j,t_k)$$
%and $v_{\rho,h}^{\eps,p}$ is defined  as following: 
%  \be\label{eq:conv} v_{\rho,h}^{\eps,p}(x,t)= h_{\eps}(x) \ast v^{p}_{\rho,h}(x,t),\ee
%with  $v^{p}_{\rho,h}(x,t)=I[v^p_{\cdot,[\frac{t}{h}]}](x)$ and
%$h_{\eps}(x)$
\be\label{eq:h_eps}
\rho(x)=\frac{1}{\sqrt{2 \pi}} e^{-x^2/ 2}.
\ee
In the numerical simulations  we approximate  \eqref{defalphaeps} with a discrete convolution, using a central difference scheme for the gradient. The iteration process  is   stopped once the quantities
%We compute the errors between two consecutive iterations of the point-fixed algorithm, measuring the discrete norms at a iteration $p$:
\be\label{errors}E(v^{\eps,p}):= \|v^{\eps,p+1}-v^{\eps,p}\|_{\infty},\quad E(m^{\eps,p}):= \|m^{\eps,p+1} -m^{\eps,p}\|_{\infty},\ee
are below a given threshold $\tau$ or when it has reached a fixed number of iterations $p$. \smallskip
\begin{remark}
The theoretical study of the convergence of the fixed--point iterations is not analyzed  in the present paper. The analysis of a convergent and efficient method  to solve \eqref{MFGfully} remains  as subject of future research.
\end{remark} \smallskip

%   \be\label{eq:convdiscr} v^{\eps}_{i,k}= \rho \sum_{j\in \ZZ} h_{\eps}(x_{i}-x_j)  v_{i,k} \ee

By Proposition \ref{Lemmaboundform_hro}{\rm(ii)}, we know that $m^{\eps}$  has a compact support, uniformly in $(\eps, \rho,h)$.  Therefore, in order to calculate the iteration  $\ds m^{\eps,p+1}_{i,k}$ we only need the values  $\ds v^{\eps, p+1}_{i,k}$ for  $i$ such that $i\rho$ belongs to a compact set $K$, which is independent of $(\eps, \rho, h, p)$. This fact allows us  to drop the analysis of boundary conditions.
 
For the numerical tests we will consider running costs of the form 
$$\frac{1}{2}\alpha^2(t)+F(x,m(t))=\frac{1}{2}\alpha^2(t)+f(x)+ V(x,m(t)),$$
where $f$ is $C^2$ and 
\begin{equation}\label{eq:V}
V(x,m(t))= \rho_{\sigma} \ast \left[\rho_{\sigma}  \ast m(t)\right](x), \hspace{0.3cm} \mbox{for some $\sigma>0$ to be chosen later.}
\end{equation}
%\be
%\label{eq:h_eps}
%h_{\sigma}(x)=\frac{1}{\sqrt{2 \pi \sigma} } e^{-(x-0.5)^2/ 2\sigma^2}.
%\ee
A straightforward calculation shows that   $F(x,m(t))=f(x)+ V(x,m(t))$  satisfies assumption {\bf (H4)}.
\subsection{Test 1}
We simulate a game where the agents  are adverse to the presence of other agents  during the game and, at the end, they do not want to live near the boundary of   a domain $\Omega$.
In order to model this situation, we take   $\Omega=[-0.1,1.1]$, and    running cost 
 $$\frac{1}{2}\alpha^2+F(x,m )=\frac{1}{2}\alpha^2+0.3V(x,m),$$
where $V$ is given by \eqref{eq:V} with  $\sigma=0.2$.  We choose $T=1$ as final time and 
$$G(x)=-0.5(x+0.5)^2(1.5-x)^2,$$
as final cost function. We take as initial mass distribution  
$$m_0(x)= \frac{\nu(x)}{\int_{\Omega}\nu(x)dx}$$ where $\nu(x)=\mathbb{I}_{[0,1]}(x)(1-0.2\cos(\pi x))$.

The second term in the definition of  $F$ penalizes  high mass density during the game whereas the final  condition $G$ 
 penalizes the fact that the agents are near   the boundary at time $T$.
 
We consider two series of numerical tests for a better understanding of the role of the regularizing parameter $\varepsilon$.
In the first series, we fix smalls space and time steps and we  vary the regularization parameter $\eps$. In the second series we decrease all the parameters $(\eps,\rho,h)$, respecting the balancing rules in Theorem \ref{resultadoprincipalfullydiscrete}.

Fig. \ref{Test1vareps} shows the  behavior of the errors \eqref{errors} in logarithmic scale on the $y$-axis versus the number of fixed--point iterations on the $x$-axis. 
We fixed    $\rho=0.0075$, $h=0.015$ and computed  20 fixed-point iterations for each one of the following   values of $\eps$:    $\eps=0.4$, $\eps=0.04$ and   $\eps=0.004$. 
 We observe a slower convergence when  $\varepsilon=0.4$ and we get a better and very similar  result when $\varepsilon=0.04,0.004$. %Although, the  best choice for $\eps$ is $0.04$ which is the one suggested by
%when both errors in  \eqref{errors} are below $\tau=10^{-3}$.
%The behavior of the errors shows  linear convergence to the solution of the coupled system.
%The fixed point iteration method has been stopped when both errors in  \eqref{errors} are below $\tau=10^{-3}$.
%Fig.\ref{Test1epsvar} shows the behavior of the errors \eqref{errors} for the choices $\varepsilon=0.4,0.04,0.004$  and $\rho=0.0075$ and $ht=0.015$. 
% which is  the one suggested from the theoretical balance $\eqref{BR}$ with $C=1$.
%\be\label{BR}\varepsilon\geq\left(\frac{\rho^2}{C}\right)^{\frac{1}{3}}.\ee
\begin{figure}[ht!]
\begin{center}
\includegraphics[width=5cm]{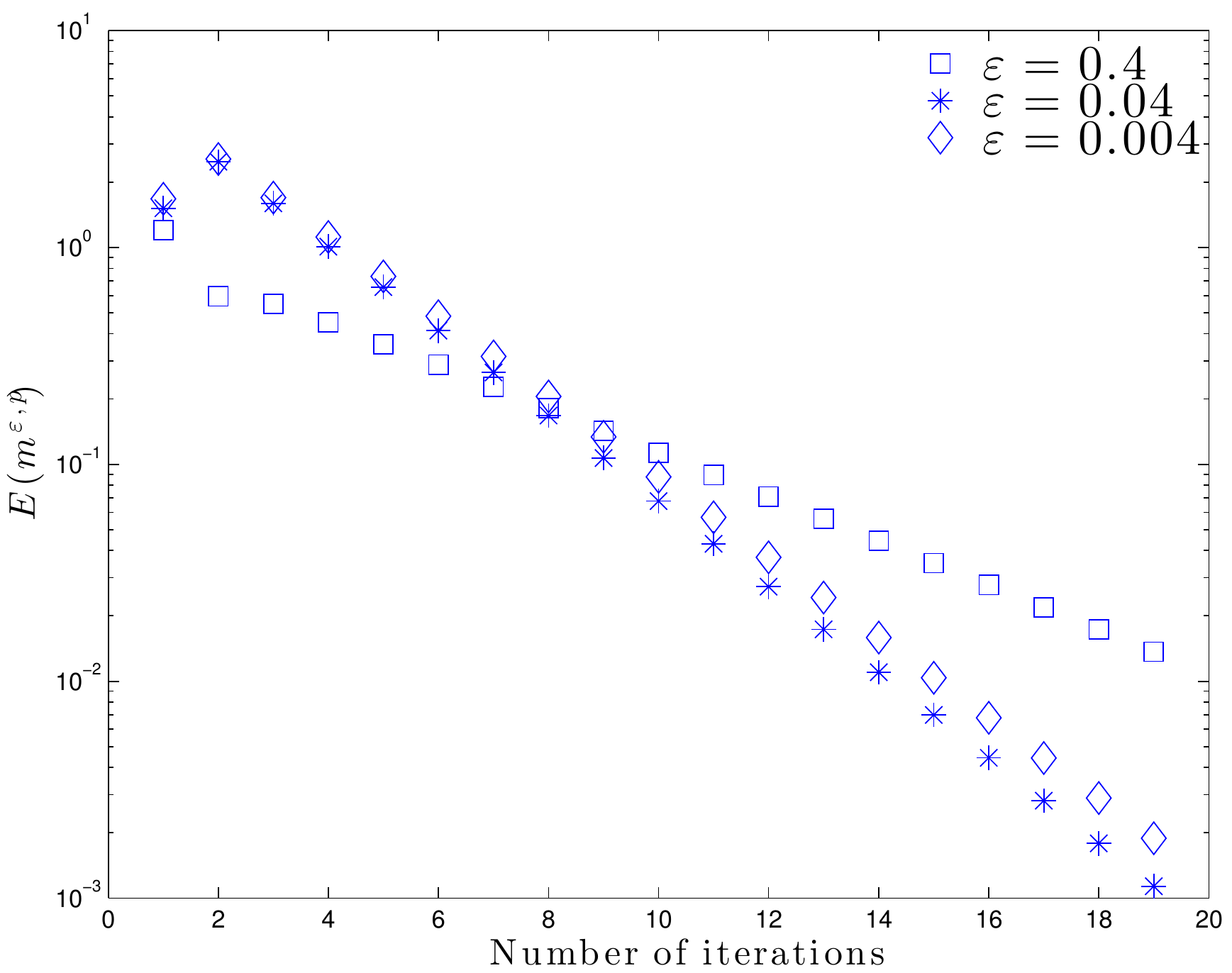}
\includegraphics[width=5cm]{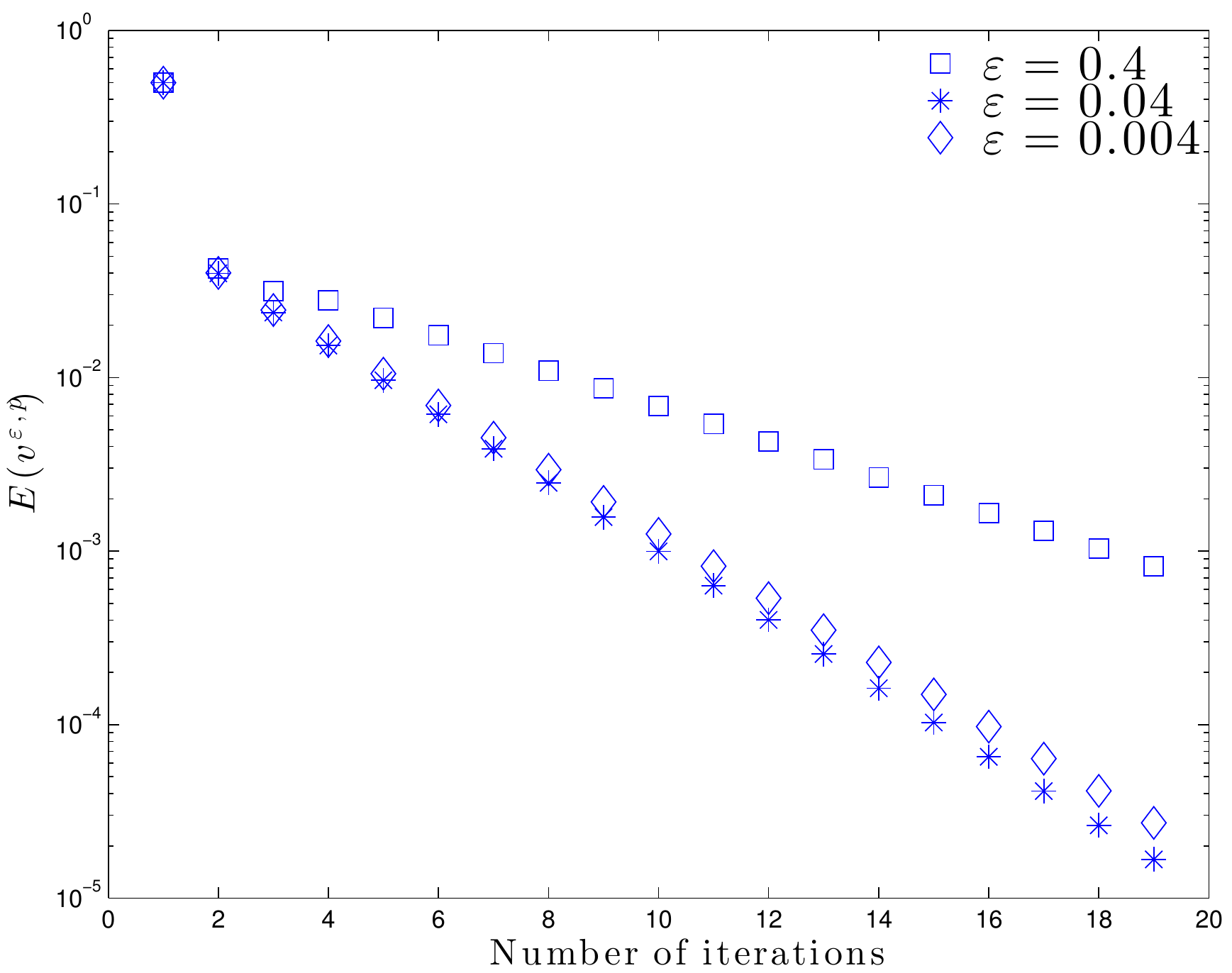}
\caption{ {\bf Errors}: $E(m^{\eps,p})$, $E(v^{\eps,p})$ varying only $\varepsilon$ and  keeping the other parameters fixed.} 
\label{Test1vareps}
\end{center}
\end{figure}
\begin{table}[ht!!]\caption{Parameters and errors }
\begin{center}\label{tab:test1SL}
\begin{tabular}{|c|c|c|c|c|}\hline
$\rho$ & $h$ & $\varepsilon $ &$E(v^{\eps,20})$& $E(m^{\eps,20})$ \\
\hline \hline
$1.50\cdot 10^{-2}$&$3.00 \cdot 10^{-2}$& $6.00 \cdot 10^{-2} $& $4.57 \cdot 10^{-6} $& $2.08 \cdot 10^{-4} $\\ \hline 
$7.50\cdot 10^{-3}$&$1.50 \cdot 10^{-2}$& $4.00 \cdot 10^{-2}$ & $1.05 \cdot 10^{-5} $&  $7.20 \cdot 10^{-4} $ \\ \hline 
$3.75\cdot 10^{-3}$&$7.50 \cdot 10^{-3}$& $2.50 \cdot 10^{-2} $& $1.04\cdot 10^{-5} $ & $9.96\cdot 10^{-4} $\\ \hline  
$1.87\cdot 10^{-3}$&$3.75 \cdot 10^{-3}$& $1.60 \cdot 10^{-2} $& $9.74 \cdot 10^{-4} $ &$3.56\cdot 10^{-3} $\\ \hline 
\end{tabular}\\[30pt]
\end{center}
\end{table}
\begin{figure}[ht!]
\begin{center}
\includegraphics[width=5cm]{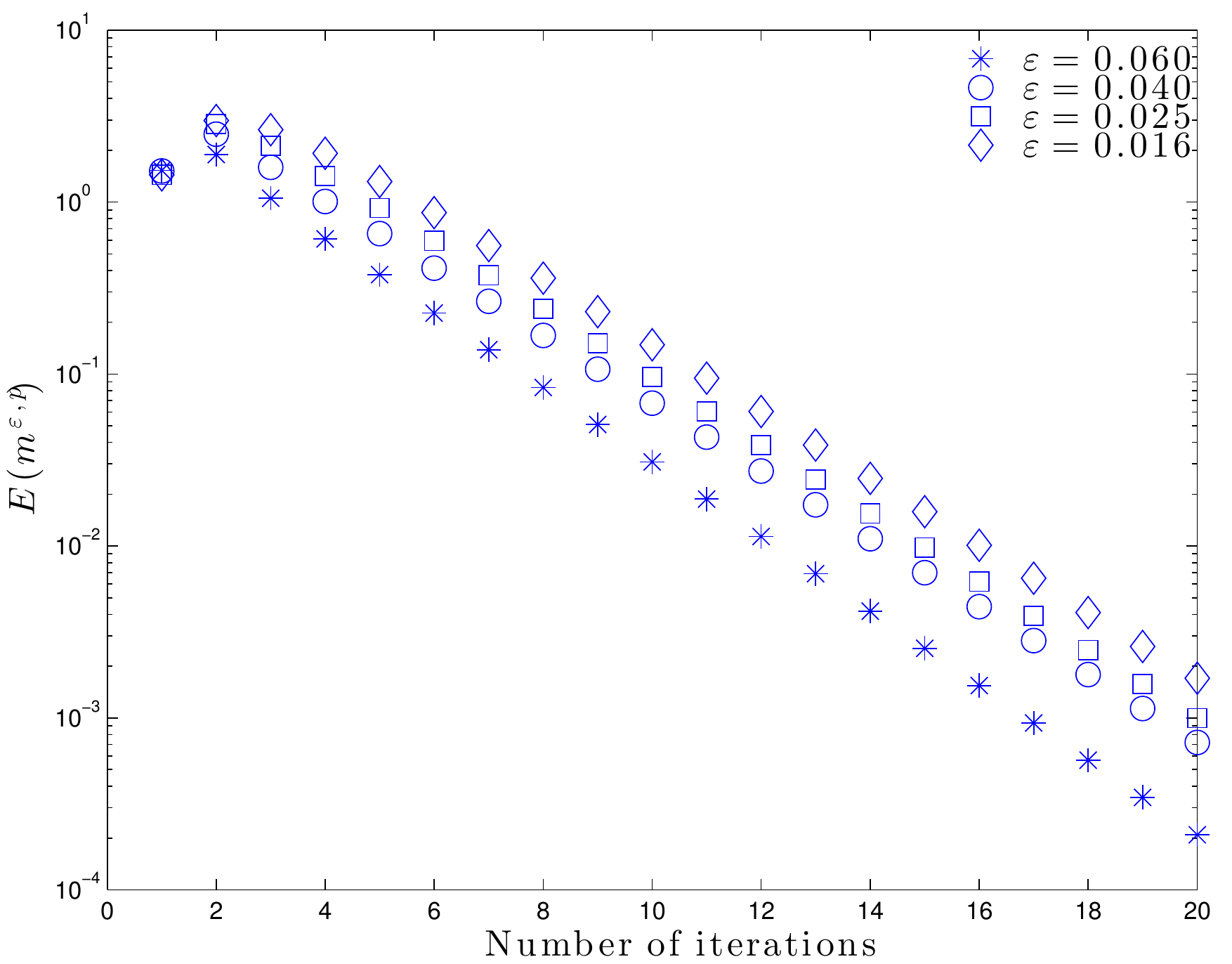}
\includegraphics[width=5cm]{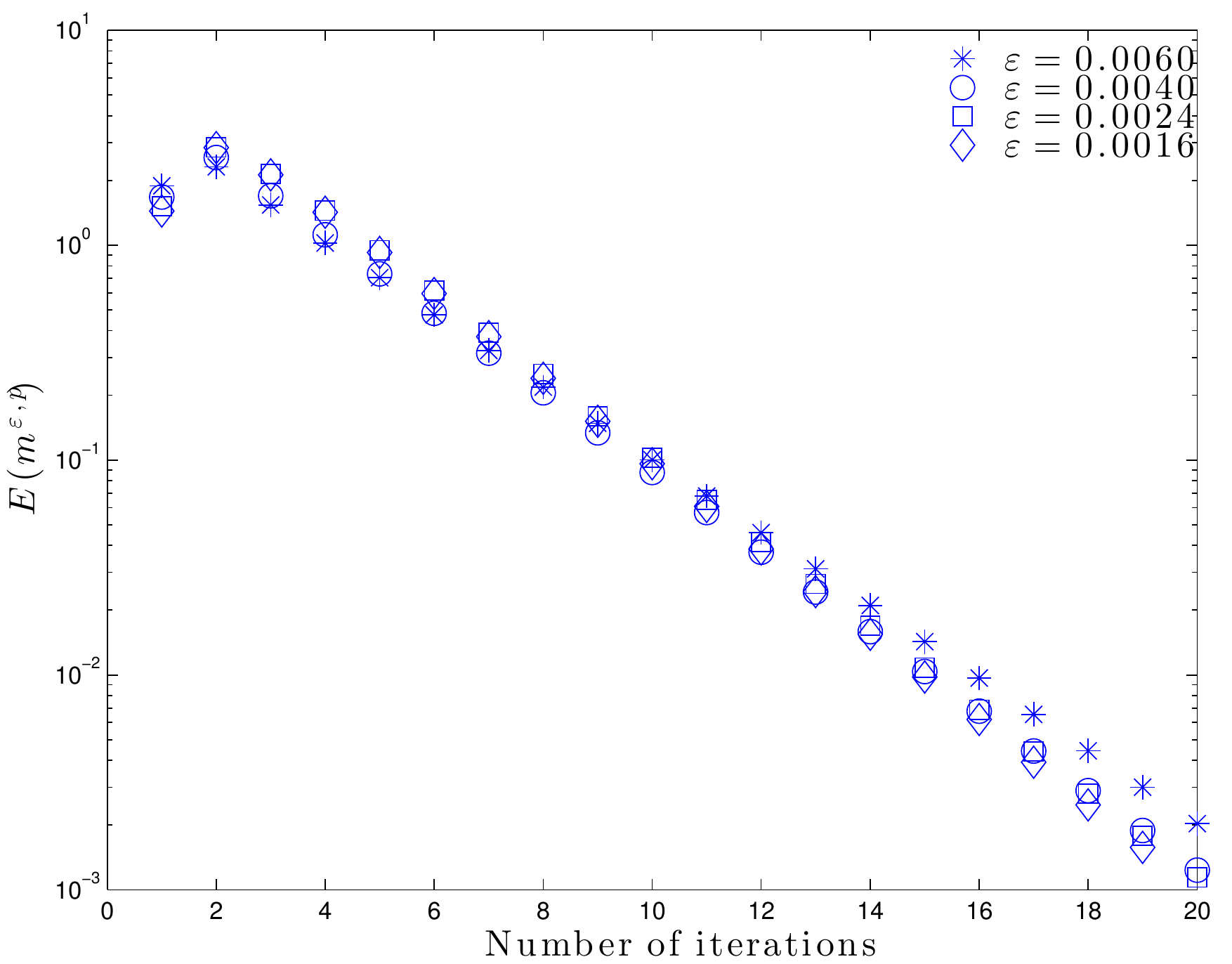}
\caption{ {\bf Errors}: $E(m^{\eps,p})$ varying all the parameters $(\eps,\rho,h)$ according to  Table \ref{tab:test1SL} (left), $E(m^{\eps,p})$  varying  $(\rho,h)$ according to  Table \ref{tab:test1SL} and varying  the regularizing parameter $\eps$ with smaller values (right).\label{Test1parvar}}
\end{center}
\end{figure}\\
In the second series of numerical tests, we vary all the parameters as shown in Table \ref{tab:test1SL} and for each set of parameters we computed $20$ fixed-point iterations.
The parameters have been chosen according to the balance requirements of Theorem \ref{resultadoprincipalfullydiscrete}.
We   observe an increasing trend in the errors with respect to decreasing values of $\eps$. This is due to the fact that we fixed the number of iterations and that smaller are the discretization parameters greater are the number of iterations to reach the fixed error threshold $\tau$.

% making two choices for the constant $C$ : $C=1$ and $C=1000$.  
%Table shows the values of the parameters for the two cases, $C=1$ and $C=1000$.
%This two choise  the second one make  the parameter $\eps$ small with respect $(ht,\rho)$ .
In Fig.\ref{Test1parvar} (left), we plot  the errors for the mass distributions with all the  parameters $(\rho,h,\varepsilon)$ varying  as in Table \ref{tab:test1SL}. 
%In both cases, $C=1$ and $C=1000$, 
We can see that the errors of the fixed point algorithm decreases with the number of fixed point iterations $p$.  

Let us remark that the  theoretical balance of parameters in  Theorem \ref{resultadoprincipalfullydiscrete} requires  to choose the regularizing parameter $\eps$ quite large compared to the space step $\rho$.
However, even disregarding this request and choosing  to regularize less, i.e. taking $\eps$ smaller,  we still get the numerical  convergence. This is shown 
in Fig.\ref{Test1parvar} (right), where we plot  the errors for the mass distributions with the  parameters $(\rho,h)$ varying  as in the first two columns of Table \ref{tab:test1SL} and setting $\eps$ on each row from the top to the bottom equal to $\eps=6.00 \cdot 10^{-3}$, $4.00 \cdot 10^{-3}$, $2.50 \cdot 10^{-3}$ and $1.60 \cdot 10^{-3}$, respectively.
%We observe that we get, with the same number of iterations, larger errors with respect the case $\eps$ larger.\\

In all the tests, we observe the same shape for the mass and value function evolution. 
In Fig. \ref{Test1mass} we plot the mass evolution in the  time--space domain $\Omega \times[0,T]$ for the case $\rho=3.75\cdot 10^{-3}$, $h=7.5 \cdot 10^{-3}$ and $\varepsilon=0.025$. 
We  observe that from the initial configuration, the mass distribution tends to avoid the boundary of $\Omega$  and at the same time it does not accumulate completely at the center. 

%, where the .
% Since the initial configuration is not symmetric, so it is the final configurations.
In Fig. \ref{Test1valuefandoptf}  the discrete value function  $v^{\eps}_{i,k}$ and its  gradient  $D v^{\eps}_{i, k}$ are plotted in the domain $\Omega\times[0,T]$.
%where $\widetilde{m}^n(x_j)=\frac{1}{2l+1}  \underset{i: |i-j|<=l} \sum m^n(x_i)$ with $l=4$.

%We compute the errors between two consecutive iterations of the point-fixed algorithm, measuring the discrete norms at a iteration $p$:
%$$E(v^p):= \|v^p-v^{p-1}\|_{\infty}\quad E(m^{\eps,p}):= \|m^{\eps,p} -m^{\eps,p-1}\|_{1}$$ 
%where    $\|f -g\|_{\infty}:=\max_{i,j}|f_{i,j}-g_{i,j}| $ and $\|f -g\|_{1}:=\rho h \sum_{i,j}|f_{i,j}-v_{i,j}| $ , where $f,g \in \B(\G_{\rho,h})$. \\
%\begin{figure}[ht!]
%\begin{center}
%\epsfig{figure=test1errinfmassa-eps-converted-to.pdf,width=5cm}\epsfig{figure=test1erru-eps-converted-to.pdf,width=5cm}
%\caption{ {\bf Errors}: $E(m^{\eps,p})$ {\rm (left)}, \quad $E(v^{\eps,p})$ {\rm (right)}, $\;p=0....,20$}
%\label{Test1error}
%\end{center}
%\end{figure}

\begin{figure}[ht!]
\begin{center}
\includegraphics[width=3cm]{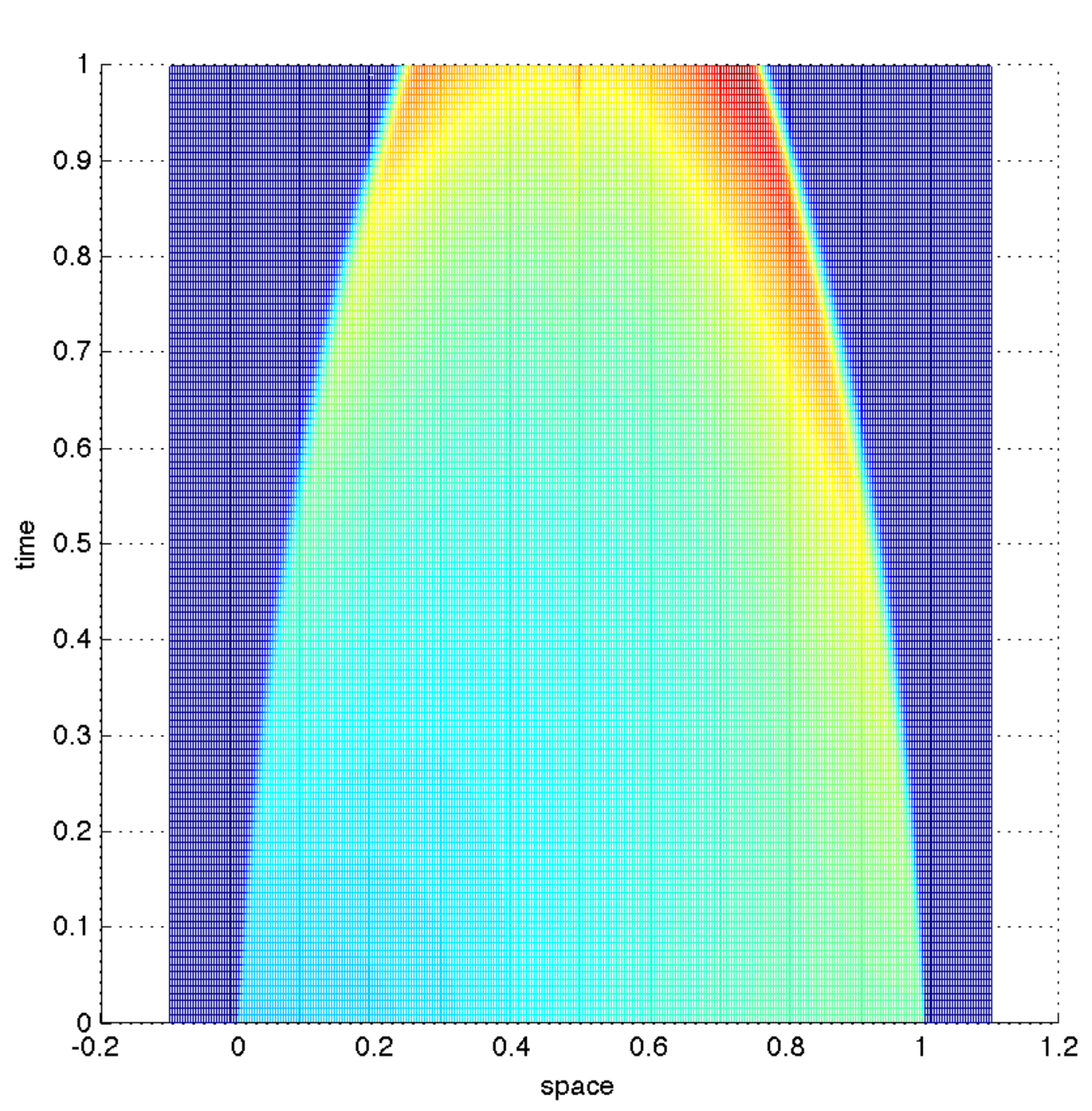}
\includegraphics[width=5cm]{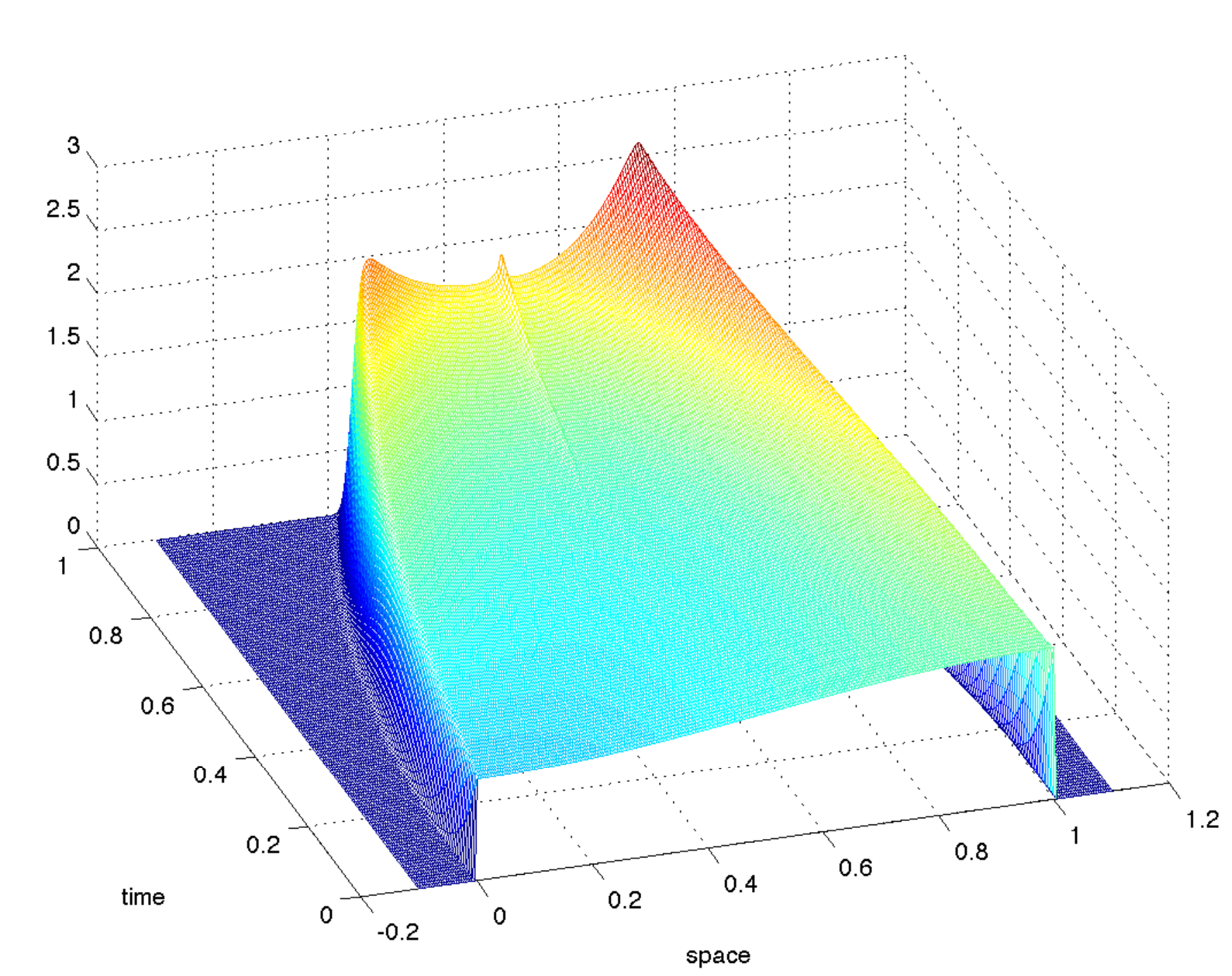} 
\caption{{\bf Mass evolution $m_{i,k}^\eps$}}
\label{Test1mass}
\end{center}
\end{figure}
\begin{figure}[ht!]
\begin{center}
\centerline{\includegraphics[width=5cm]{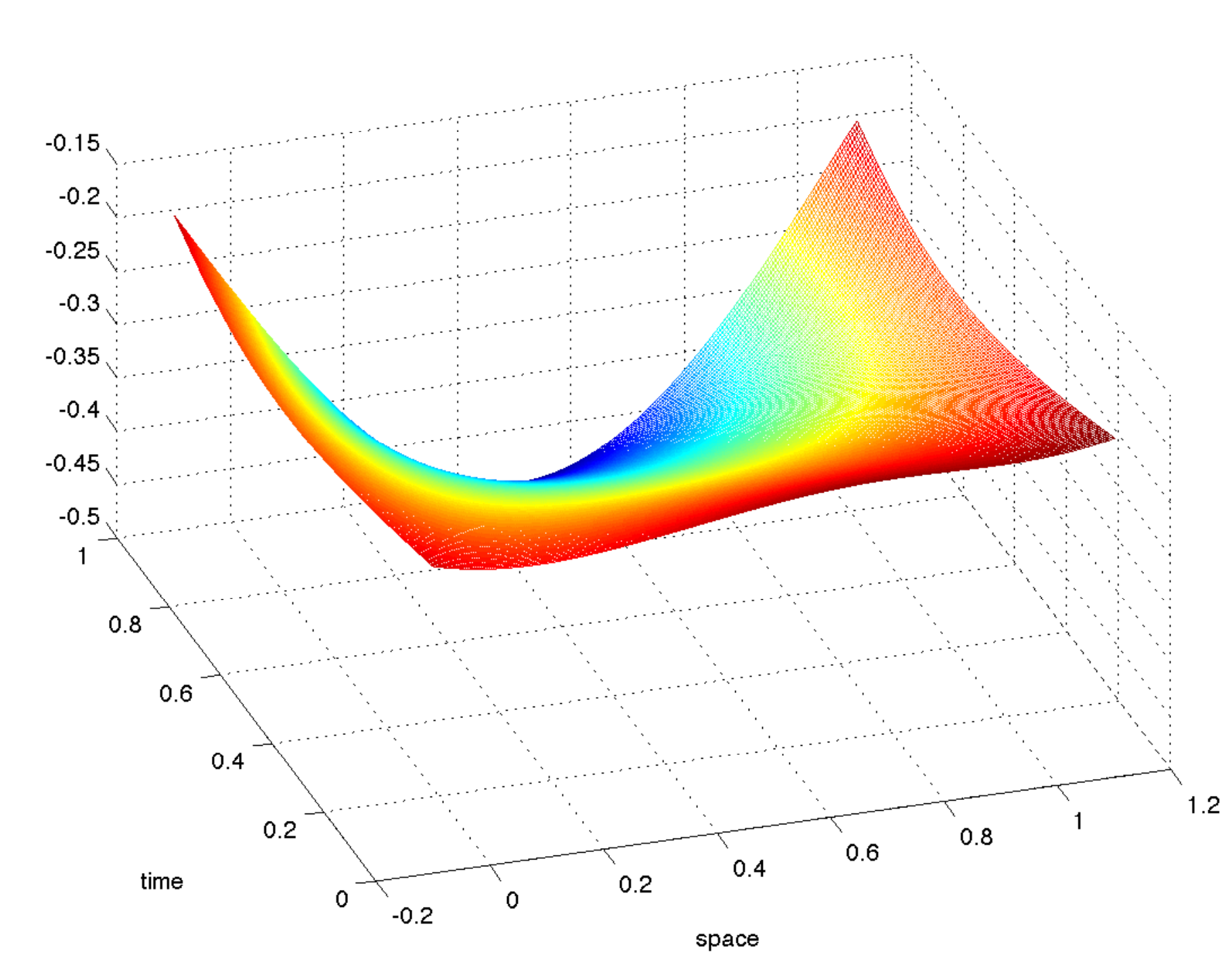}
\includegraphics[width=5cm]{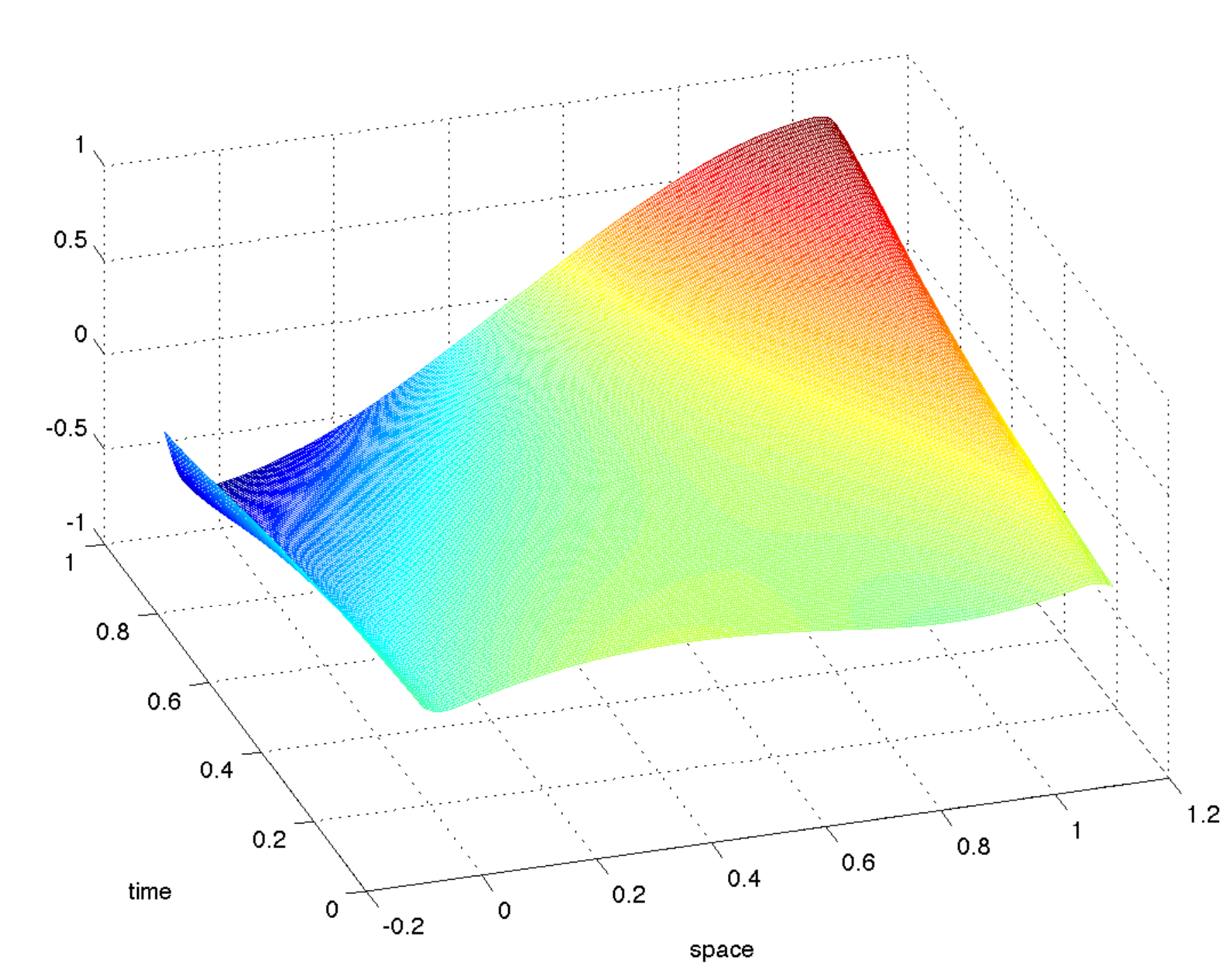}}
\caption{ {\bf{ Value function }} $v^{\eps}_{i,k}$ (left) and {\bf Gradient } $D v^{\eps}_{i,k}$(right). }
\label{Test1valuefandoptf}
\end{center}
\end{figure}
\subsection{Test 2}  We model now a game where the agents want to live at $x=0.2$ but again they are adverse to the presence of other agents.
We take as  numerical  space domain  $\Omega= [0,1]$ and    final time $T=1$. The running cost   function is modeled as $$\frac{1}{2}\alpha^2+F(x,m )=\frac{1}{2}\alpha^2+(x-0.2)^2+ V(x,m),$$ 
where $V(x,m)$ is defined in \eqref{eq:V} with  $\sigma=0.25$. We do not consider a final cost, i.e. we take $G\equiv 0$. 
We choose  as initial mass distribution: 
$$m_0(x)=\frac{\nu(x)}{\int_{\Omega}\nu(x)dx},\; {\rm with}\; \nu(x)= e^{-(x-0.75)^2/ (0.1)^2}.$$
We choose as space discretization step  $\rho=3.3\cdot 10^{-3}$, as time step  $h=0.005$ and as regularization parameter $\eps=0.025$. We perform fixed-point iterations until the error threshold $\tau=10^{-3}$ is reached. This is achieved after $15$ iterations.
%At the initial time the mass is concentrated in $x=0.75$; due to the first term in the running cost the mass tends to move to the point $x=0.2$, the second term in $F$ force to keep $V$ not to high, i.e. the agents do not want to be too much together.\\
%The simulation shows that some of the agents do not move at the beginning: since moving is expensive, there is a part of the agents that do not move immediately, only after some time they reach the point $x=0.2$.
%It is also possible to observe that, again due to the second term in $F$, some of the agents just move away from the other reaching the boundary on the left.\\
\begin{figure}[htbp]
\begin{center}
\includegraphics[width=3cm]{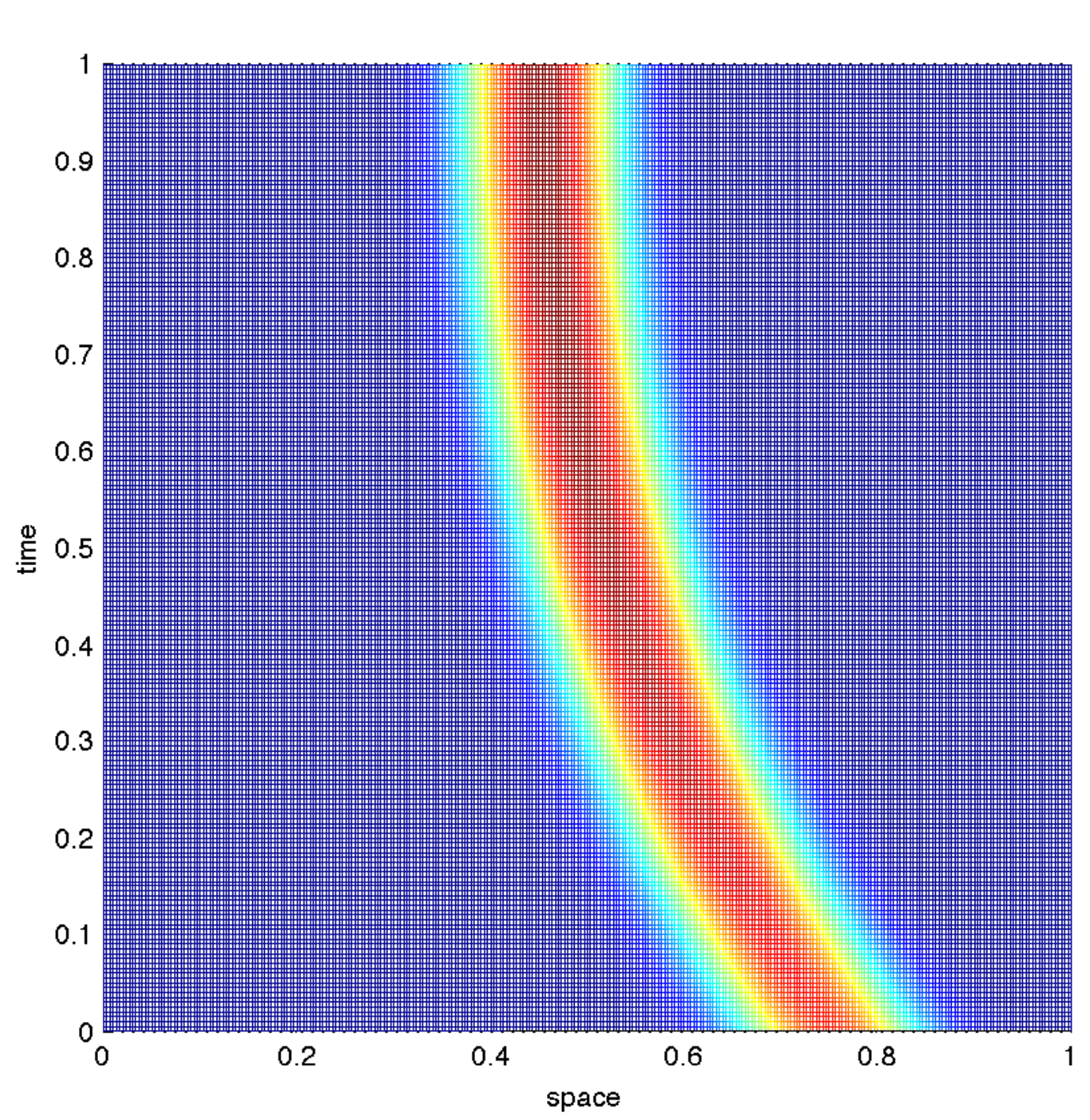} 
\includegraphics[width=5cm]{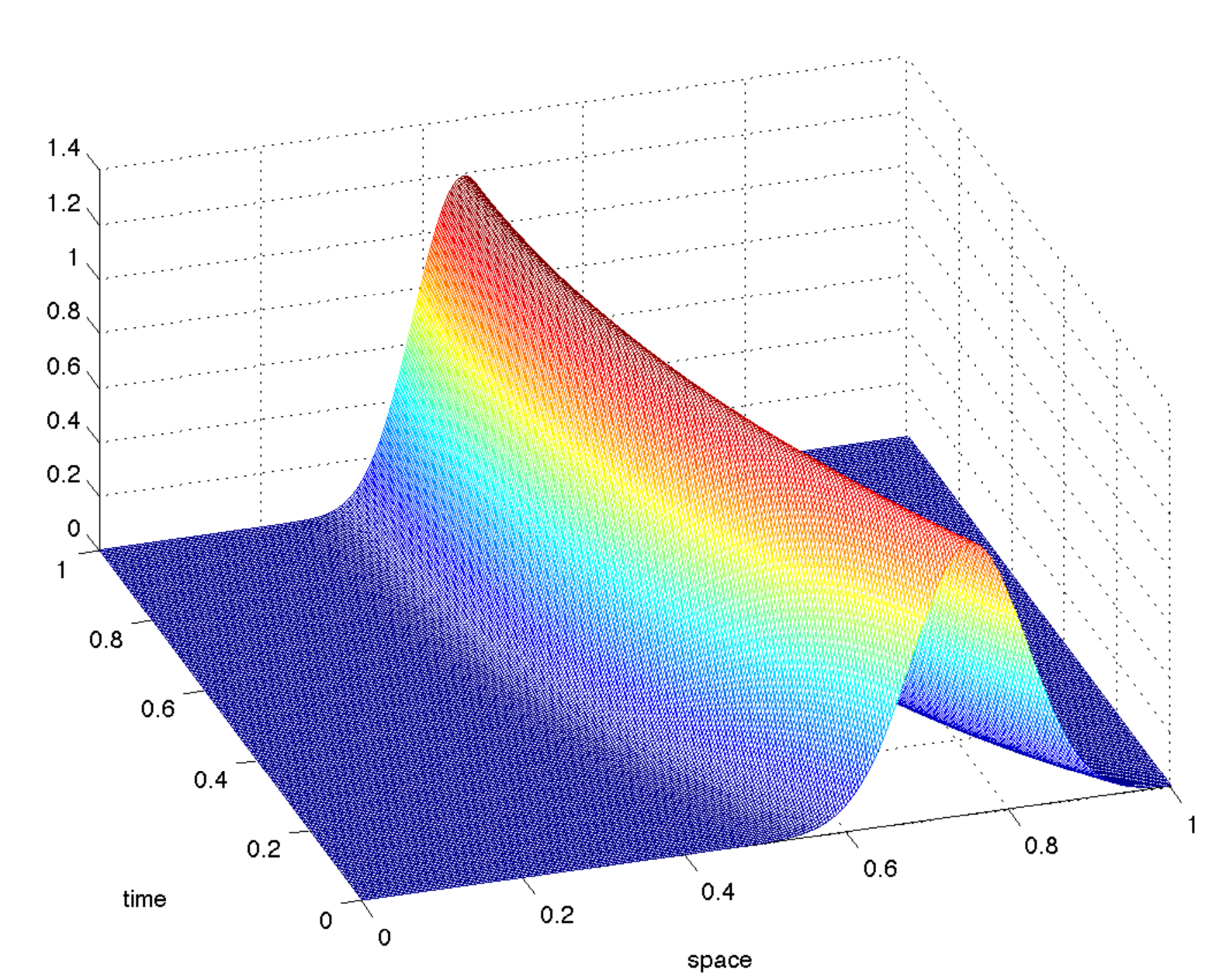} 
\caption{{\bf Mass distribution $m^{\eps}_{i,k}$}.}
\label{Test2mass}
\end{center}
\end{figure}
Fig. \ref{Test2mass} shows the mass evolution.  As it is expected, during the  evolution the mass distribution tends to concentrate  at the ``low energy'' configuration   $x=0.2$ and at the same time the  second term in $F$  penalize  high  mass concentrations. The discrete value function  and its gradient are plotted in   Fig. \ref{Test2valuefunction} and in Fig. \ref{Test2error}  we display the errors $E(m^{\eps,p})$ and $E(v^{\eps,p})$ of the fixed--point iterations. 
%The behavior of the errors shows  linear convergence to the solution of the coupled system.
%The fixed--point iteration has been stopped when both the errors are below .\\
% Clearly,at the beginning more effort is required  to move the mass in the point $x=0.2$. After some time, since already mostly of the people have reached the low cost position, the flow gets smooth.\\
\begin{figure}[ht!]
\begin{center}
\includegraphics[width=5cm]{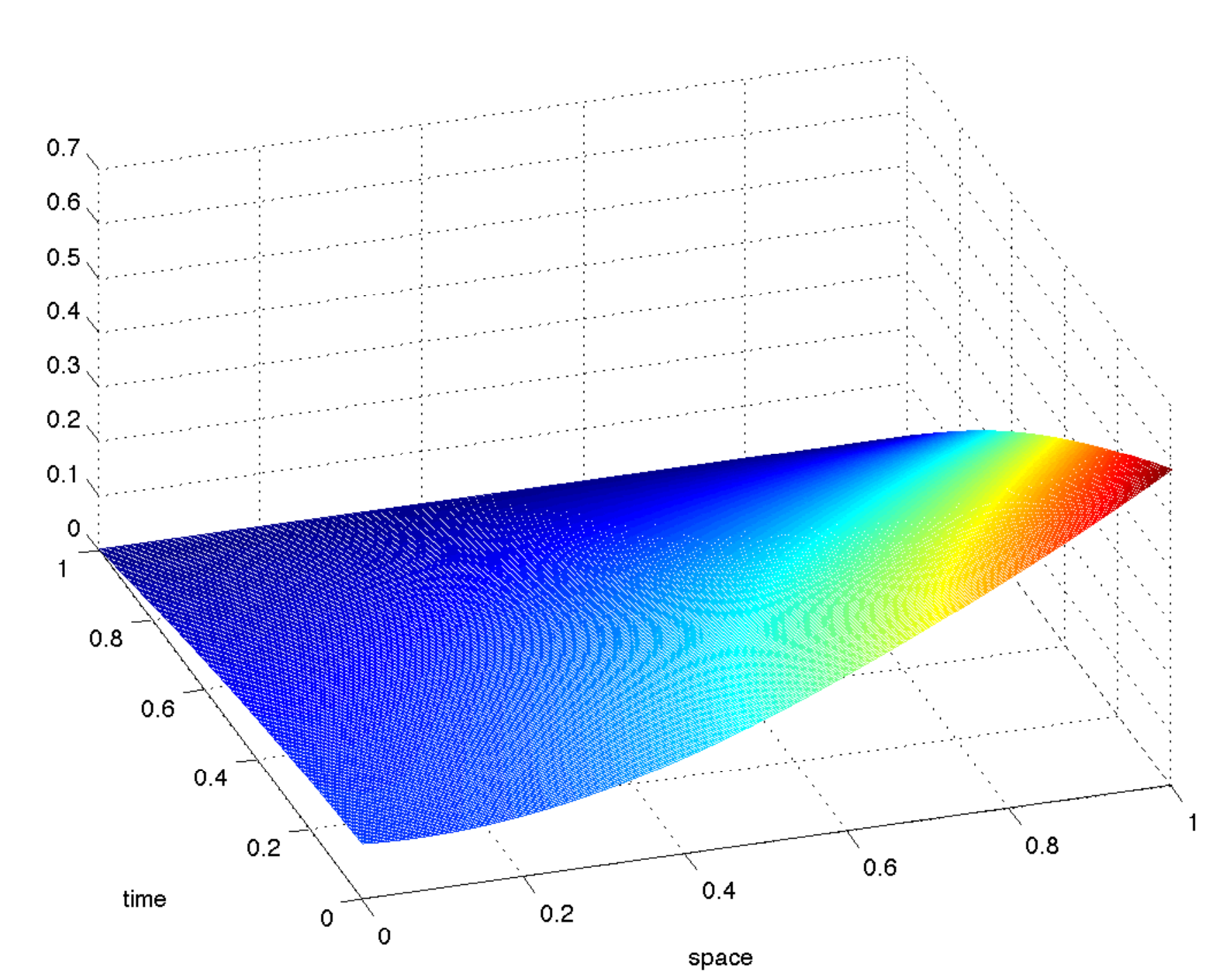}
\includegraphics[width=5cm]{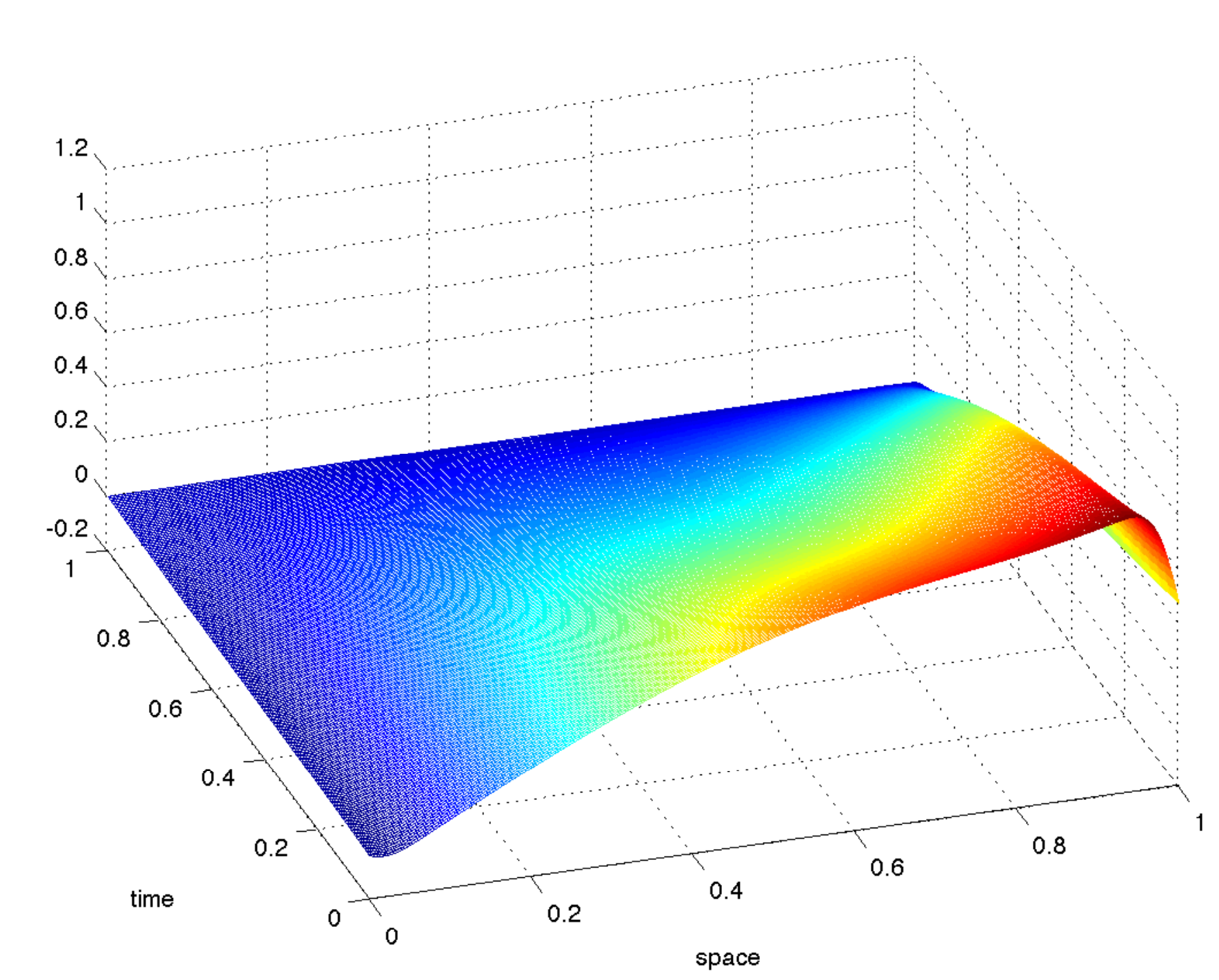}
\caption{ {\bf{ Value function }} $v^{\eps}_{i,k}$ (left) and {\bf Gradient } $D v^{\eps}_{i,k}$(right). }
\label{Test2valuefunction}
\end{center}
\end{figure}

\begin{figure}[ht!]
\begin{center}
\includegraphics[width=5cm]{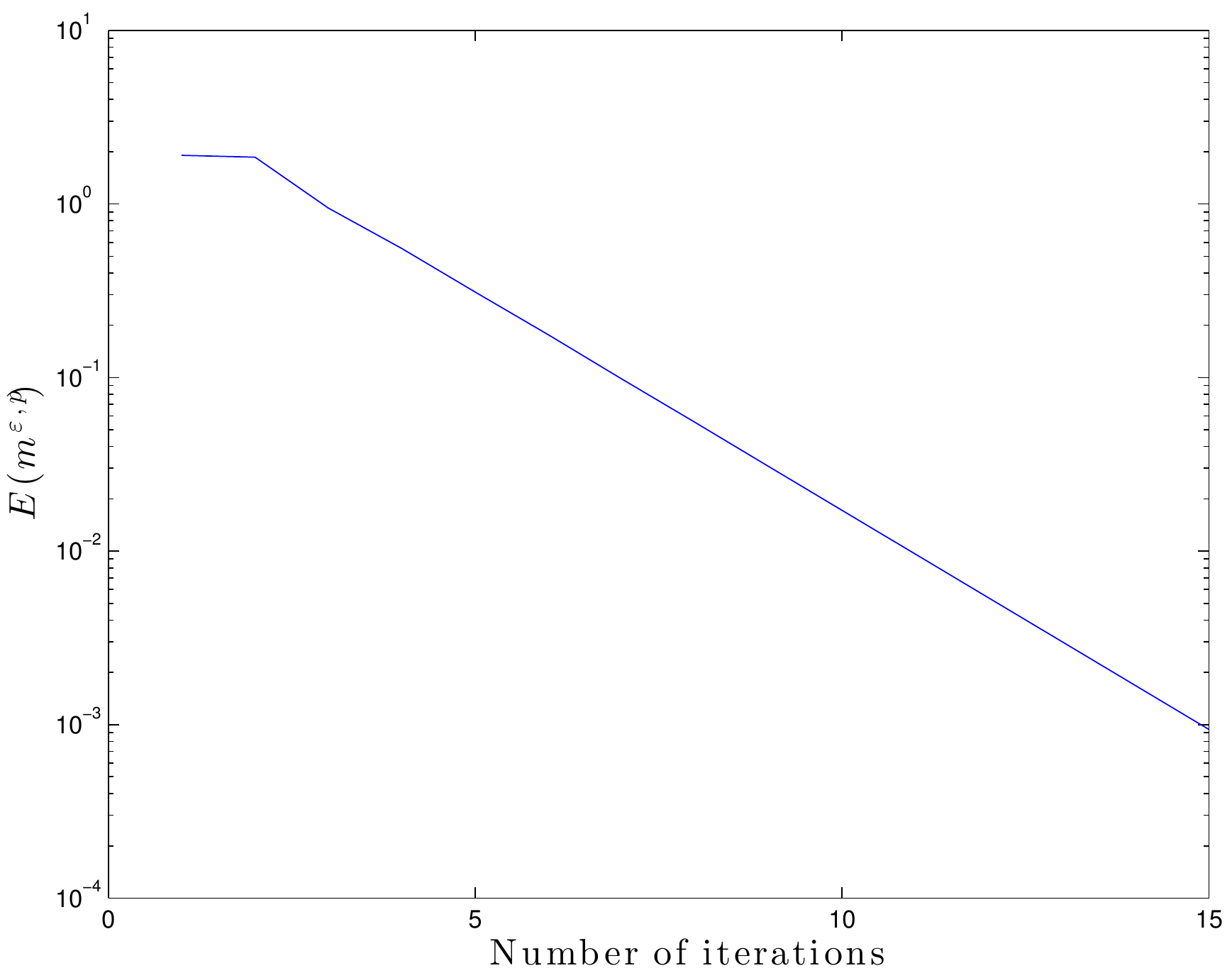}
\includegraphics[width=5cm]{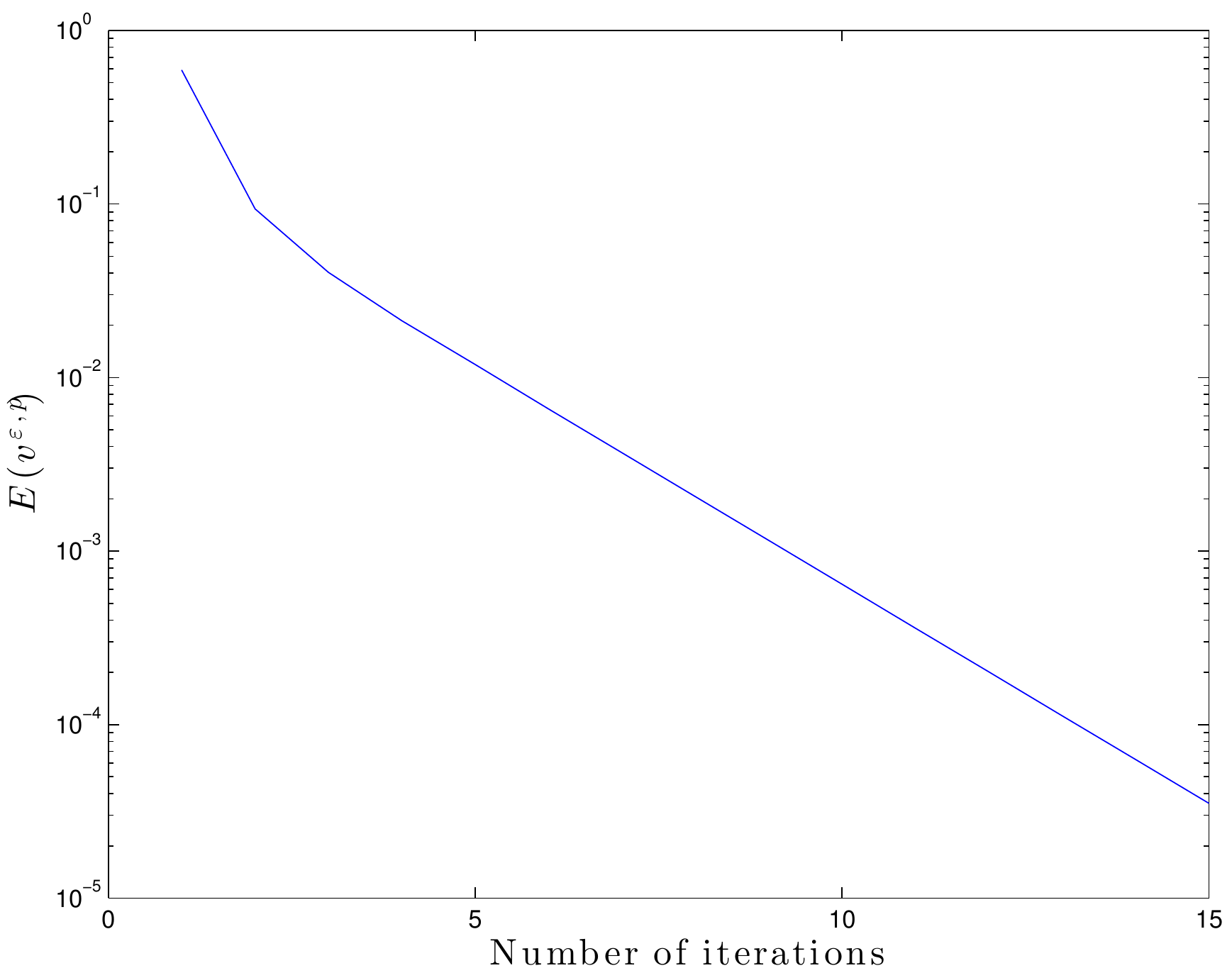}
\caption{ {\bf Errors}: $E(m^{\eps,p})$ {\rm (left)}, \quad $E(v^{\eps,p})$ {\rm (right)}, $\;p=0....,15$.}
\label{Test2error}
\end{center}
\end{figure}
Let us finally compare this test to the case when there is no game, i.e. the running cost does not depend on $m$:
 $$F(x,m )=(x-0.2)^2.$$
 In this case, the system is not coupled and   after one iteration we obtain the solution.  \\
 In Fig. \ref{Test2massb}, the mass evolution is shown. It is seen that, during the evolution,  the measure is allowed to concentrate, due to the absence of a high mass penalization term in $F$. This  shows   qualitative  differences with the results plotted in   Fig. \ref{Test2mass}, where conflict between the agents  was present.
% 
%  One can note that difference between observe that  the agents reach the point $x=0.2$ very fast compared to the evolution shown in Fig. \ref{Test2mass}, keeping mostly the same mass distribution during the evolution; this is due to the fact that in this case the function $F$ does not depend on $m^\eps_{\rho,h}$.
 %,  compared to the evolution shown in Fig. \ref{Test2mass}.
 \begin{figure}[htbp!]
\begin{center}
\includegraphics[width=3cm]{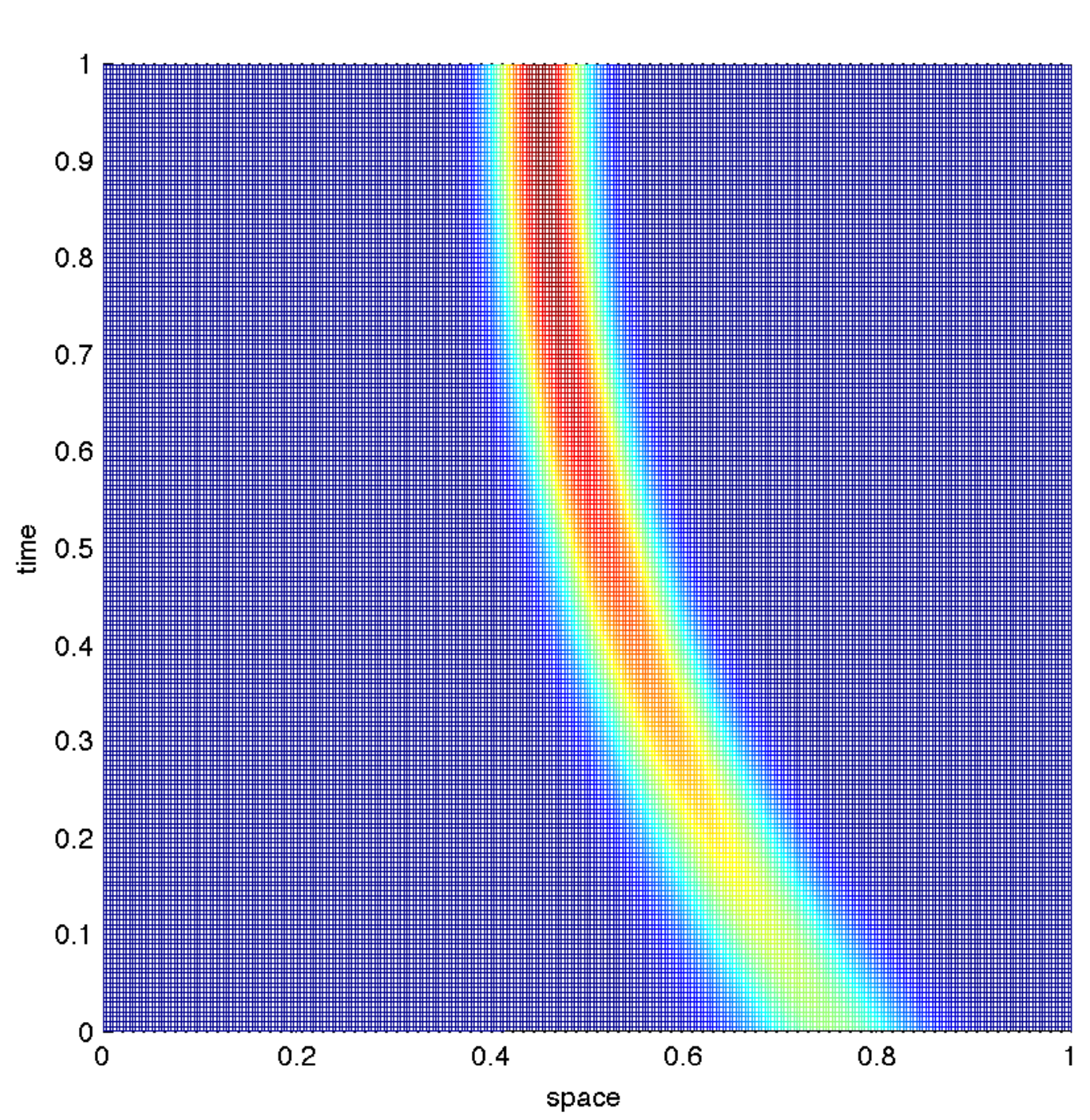} 
\includegraphics[width=5cm]{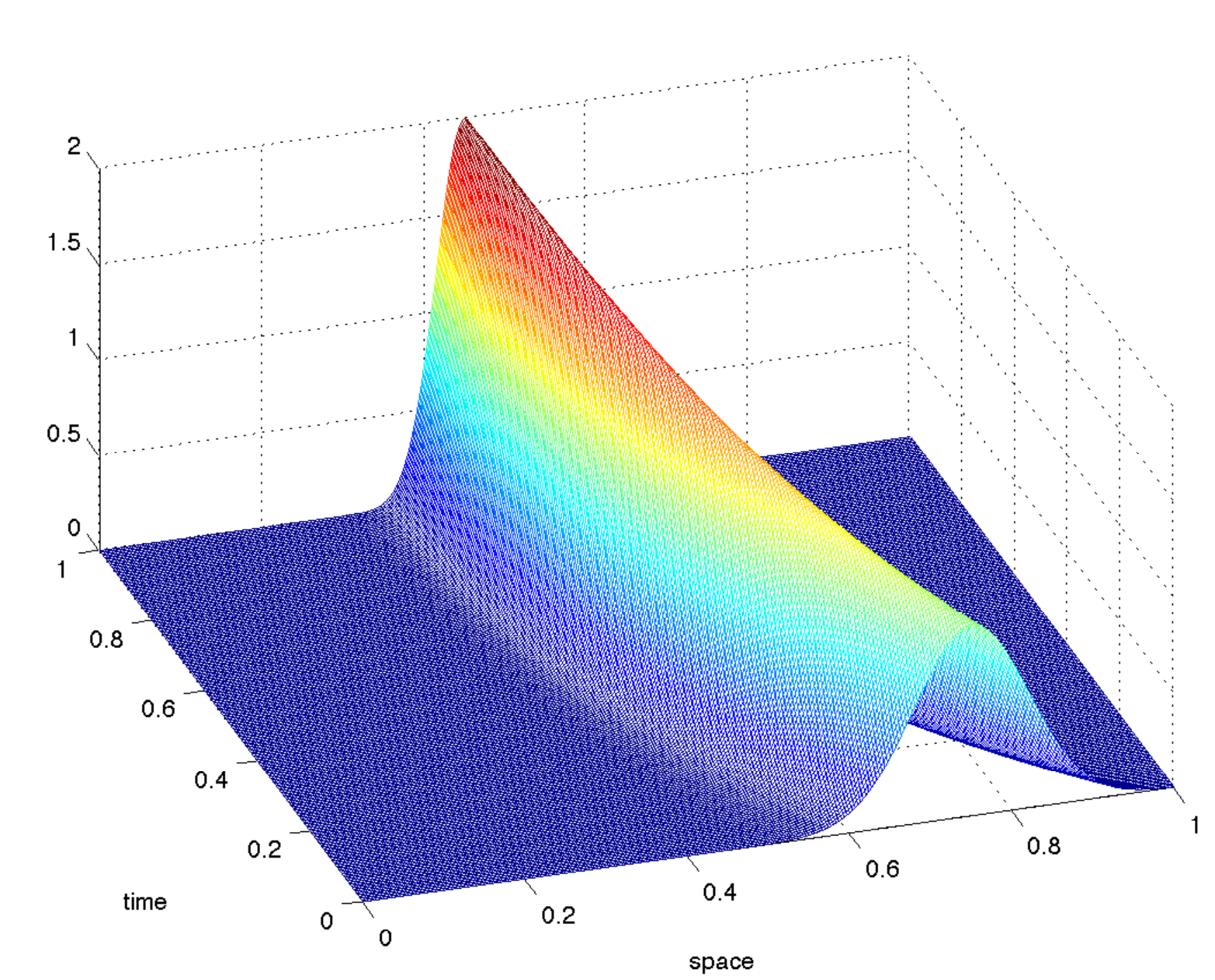} 
\caption{{\bf Mass distribution} $m^{\eps}_{i,k}$ (case ``no game'' with $F=(x-0.2)^2$ ).}
\label{Test2massb}
\end{center}
\end{figure}
\bibliographystyle{plain}
\bibliography{bibpostdocultimo}
 %\end{thebibliography}
\end{document}